\documentclass[11pt]{amsart}
\usepackage{amsmath, amsthm, amssymb}
\usepackage{graphicx}
\usepackage[usenames]{color}
\usepackage{srcltx} 
\usepackage{verbatim}
\usepackage{tikz}
\usepackage{alltt}
\usetikzlibrary{shapes}
\usetikzlibrary{plotmarks}

\newtheorem{thm}{Theorem}[section]
\newcommand{\bt}{\begin{thm}}
\newcommand{\et}{\end{thm}}

\newtheorem{ex}[thm]{Example}

\newtheorem{cor}[thm]{Corollary}   
\newcommand{\bc}{\begin{cor}}
\newcommand{\ec}{\end{cor}}

\newtheorem{lem}[thm]{Lemma}   
\newcommand{\bl}{\begin{lem}}
\newcommand{\el}{\end{lem}}

\newtheorem{prop}[thm]{Proposition}
\newcommand{\bp}{\begin{prop}}
\newcommand{\ep}{\end{prop}}

\newtheorem{defn}[thm]{Definition}
\newcommand{\bd}{\begin{defn}}    
\newcommand{\ed}{\end{defn}}

\newtheorem{rmrk}[thm]{Remark}   

\newcommand{\br}{\begin{rmrk}}
\newcommand{\er}{\end{rmrk}}

\newcommand{\mGHto}{\xrightarrow{\, \textrm{mGH}\,} }
\newcommand{\GHto}{\stackrel {\textrm{GH}}{\longrightarrow} }
\newcommand{\Fto}{\stackrel {\mathcal{F}}{\longrightarrow} }
\newcommand{\VFto}{\stackrel {\mathcal{VF}}{\longrightarrow} }

\newcommand{\paeto}{\stackrel {\text{p.a.e.}}{\longrightarrow} }

\newcommand{\be}{\begin{equation}}

\newcommand{\ee}{\end{equation}}

\newcommand{\N}{\mathbb{N}}

\newcommand{\R}{\mathbb{R}}

\newcommand{\diam}{\operatorname{Diam}}

\newcommand{\disjointunion}{\sqcup}

\newcommand{\mass}{{\mathbf M}}

\newcommand{\vol}{\operatorname{Vol}}

\newcommand{\lp}{\left (}
\newcommand{\rp}{\right )}

\newcommand{\Sp}{\mathbb{S}}      
\newcommand{\Tor}{\mathbb{T}}     

\begin{document}

\title[Relating Notions of Convergence]{Relating Notions of Convergence in Geometric Analysis}

\author{Brian Allen}
\address{University of Hartford}
\email{brianallenmath@gmail.com}

\author{Christina Sormani}
\thanks{C. Sormani was partially supported by NSF DMS 1612049.  Some of this work was completed in
discussions at the IAS Emerging Topics on Scalar Curvature and Convergence that C. Sormani co-organized with M. Gromov in 2018.}
\address{CUNY Graduate Center and Lehman College}
\email{sormanic@gmail.com}





\begin{abstract}
We relate $L^p$ convergence of metric tensors or volume convergence 
to a given smooth metric to intrinsic flat and Gromov-Hausdorff convergence for sequences of Riemannian manifolds. We present many examples of sequences of conformal metrics which demonstrate that these notions of convergence do not agree in general even when the sequence is conformal, $g_j=f_j^2g_0$, to a fixed manifold.  We then prove a theorem demonstrating that when sequences of metric tensors on a fixed manifold $M$ are bounded,  $(1-1/j)g_0 \le g_j \le K g_0$, and either the volumes converge, $\vol_j(M)\to \vol_0(M)$,
or the metric tensors converge in the $L^p$ sense, then the Riemannian manifolds $(M,g_j)$ converge in the 
measured Gromov-Hausdorff and volume preserving intrinsic flat sense to $(M,g_0)$.  
\end{abstract}

\maketitle

\section{Introduction}

There are many settings in Riemannian geometry where one must examine a sequence of Riemannian manifolds and prove that they converge in some sense to a limit space.  These situations arise when one is seeking a canonical metric in a given class, or examining how manifolds evolve under flows, or studying the stability of a rigidity theorem.  When the convergence cannot be proven to be smooth, because known examples exist which do not converge smoothly, then
one can still hope to prove convergence in some weaker sense.  Analysts will immediately be drawn to consider notions like Lesbegue ($L^p$) convergence of metric tensors.  Geometers will turn instead to geometric notions like Gromov-Hausdorff  (GH) convergence in which distances are controlled but not metric tensors.  If controlling distances seems too strong a notion, Geometers turn to Intrinsic Flat ($\mathcal{F}$) convergence, in which only the filling volumes between the spaces and their limits must tend to 0.  Neither GH nor $\mathcal{F}$ convergence imply volumes of the spaces converge, and so  geometers also consider the stronger metric measure (See Sturm \cite{Sturm-mm-1} for the best definition of metric measure convergence), measured Gromov-Hausdorff distance (mGH), or volume preserving intrinsic flat $\mathcal{VF}$ convergences as well.   We review all these notions in Section~\ref{sect: review}.

Understanding compactness properties for sequences of Riemannian manifolds under various geometric conditions is a vast area of research in geometric analysis. Important work where integral bounds on curvature or volume pinching is assumed has been done by Anderson \cite{Anderson-Ricci, Anderson-Orbifold}, Anderson and Cheeger \cite{Anderson-Cheeger}, Colding \cite{Colding-volume, Colding-shape}, Gao \cite{Gao-integral1}, Petersen and Wei \cite{Petersen-Wei-integral1, Petersen-Wei-integral2}, Petersen and Sprouse \cite{Petersen-Sprouse-integral}, and Yang \cite{DYang-integral1, DYang-integral2, DYang-integral3}. In these works the authors have various results showing $C^{0,\alpha}$ (possibly away from singular points) convergence of Riemannian metrics under various integral bounds on curvature or assumptions on volume (See the survey articles of Petersen \cite{Petersen-Survey} and Sormani \cite{Sormani-Survey} for a broad discussion of convergence theorems). In the case of convergence of conformal Riemannian metrics important related work has been done by Aldana, Carron, and Tapie \cite{Aldana-Carron-Tapie}, Brendle \cite{Brendle-Conformal}, Bonk, Heinonen, and Saksman \cite{Bonk-Heinonen-Saksman}, Chang, Gursky, and Wolff \cite{Chang-Gursky-Wolff}, Gursky \cite{Gursky}, and Wang \cite{Wang-Isoperimetric}  (See the survey article by Chang \cite{Chang-Survey} for a broad look at results in conformal geoemtry). 

In this paper our goal is to provide hypotheses which can by used to bootstrap from Lebesgue convergence ($L^p$) or volume convergence to the geometric notions of convergence like GH and $\mathcal{F}$. Before we state our theorems, we describe the examples we've constructed which provide new insight into the distinctions between these notions of convergence in the conformal setting.  We consider  $M_j=(M^m, f_j^2 g_0)$ where $(M^m,g_0)$ is either
a standard flat torus or a sphere.  We begin with Example~\ref{Cinched-Sphere} in which the sequence converges in the
$L^p$ sense to the standard sphere but the GH and $\mathcal{F}$ limit is a cinched sphere.  This occurs because Lesbegue convergence cannot see what happens on a set, $S$, of measure $0$ (in this case the equator) and yet distance based notions of convergence detect short cuts through $S$ if $g_j$ is smaller than the expected $L^p$
limit on $S$.  To avoid this difficulty of having shorter paths we require $C^0$ convergence from below in the rest of our examples and our theorems.   

In Examples~\ref{No C0 Convergence}- ~\ref{PointwiseOnDenseSet}, we consider conformally flat tori $M_j=(M^m, f_j^2 g_0)$ with $f_j \ge (1-1/j)$. In Example~\ref{No C0 Convergence} we add bumps to the tori where
$f_j=K>1$ so that we have $L^p$ convergence but not $C^0$ convergence.  Here we
 see that $M_j $ converge  in the GH, $\mathcal{F}$,
mGH, and $\mathcal{VF}$ sense to $(M, g_0)$. In Example~\ref{PointwiseOnDenseSet} we have $1 \le f_j \le \sqrt{2}$ and $f_j\rightarrow 1$ pointwise on a dense set but no $L^p$ convergence. Due to the fact that more and more bumps are appearing increasingly on a finer and finer grid it becomes advantageous to travel along taxi curves and hence we find GH and $\mathcal{F}$ convergence to a torus with a taxi metric. 

In the remaining examples we explore what can happen when the uniform upper bound on the metric is removed. In these examples we see that the particular $p$ for which we have $L^p$ bounds and/or convergence of the conformal factor $f_j$ plays a crucial role. In Example~\ref{L^p Conv} we have a sequence of tori with bumps that grow taller with $j$ while keeping $L^m$ convergence of $f_j \rightarrow 1$ (where $m$ is the dimension of $M^m$).  Here we see that the volumes converge and that
we have mGH and $\mathcal{VF}$ convergence.   In Example~\ref{No L^m Conv} we allow the
bumps to grow enough that we do not have $L^m$ convergence of $f_j$ and the volumes don't converge but both do remain bounded above. In this case we find mGH and $\mathcal{VF}$ convergence to a flat torus with a flat disk attached so that the boundary of the disk is attached to a point of the flat torus. This shows that the volume convergence assumption is crucial for ruling out bubbling.  In Example~\ref{No Conv} we allow the bumps to grow enough that the $L^p$ norm of $f_j$ diverges for any $p \ge m$. In this case we see that the sequence does not converge in the GH or $\mathcal{F}$ sense to a compact metric space. This last example illustrates the worst that can happen when we do not have volume convergence. In Example~\ref{VolControlDiamNotConvergent} we allow the bumps to grow so that we have $L^m$ convergence of $f_j \rightarrow 1$ and convergence of volume, which is the borderline case between Examples~\ref{L^p Conv} and \ref{No L^m Conv}, and we see that the GH limit is a torus with a line attached. In Example~\ref{FConvNoGHConv} we adapt the previous example to allow increasingly many bumps to form while maintaining $L^m$ of $f_j \rightarrow 1$ and volume convergence so that there is no GH limit. 

Despite all these intriguing examples, we are able prove a surprisingly strong theorem 
about sequences of metrics which are all conformal to the same base manifold:

\begin{thm}\label{ConfGHandFlatConv}
Suppose we have a sequence of conformal metrics, $f_j^2(x) g_0$, $f_j: M \rightarrow (0,\infty)$ continuous, $g_0$ smooth, on a
compact oriented manifold without boundary, $M$, such that 
\be \label{conf1}
 0<1 - 1/j \le f_j \le K < \infty.
 \ee
 If we have Lebesgue convergence of the conformal factors,
  \be \label{conf-L^p}
  f_j \to 1 \textrm{ in } L^p(M, g_0) , \quad p \in [1,\infty],
  \ee
  or volume convergence,
 \be \label{conf-vol}
 \vol_{g_j}(M) \to \vol_{g_0}(M),
 \ee
     then the sequence of continuous Riemannian
   manifolds $M_j=(M, f_j^2(x) g_0)$ converges to $M_0=(M, g_0)$ in both
   the measured Gromov-Hausdorff sense and the volume preserving intrinsic flat sense: 
\be
M_j\mGHto M_0 \textrm{ and } M_j \VFto M_0.
\ee
\end{thm}

Note that Example~\ref{Cinched-Sphere} demonstrates that we cannot expect
convergence to $M_0$ in any sense without $C^0$ lower bound in the hypotheses.
Example~\ref{PointwiseOnDenseSet} demonstrates  that we cannot expect convergence to 
to $M_0$ in the mGH sense if we remove the hypothesis on $L^p$ convergence or volume convergence in 
Theorem \ref{ConfGHandFlatConv}.    Example~\ref{VolControlDiamNotConvergent} and Example~\ref{FConvNoGHConv}  show that we cannot expect the conclusion of Theorem \ref{ConfGHandFlatConv} to hold when a uniform upper bound is not assumed.

This conformal theorem is in fact a consequence of the following far more general theorem
that we prove within:

\begin{thm}\label{GenGHandFlatConvMetric}
Assume that $g_j$ is a sequence of continuous Riemannian metrics and $g_0$ is a smooth Riemannian metric defined on the same compact oriented manifold without boundary $M^m$ and 
\be
\lp 1 - 1/j\rp g_0(v,v) \le g_j(v,v) \le K g_0(v.v) \quad \forall v \in T_pM.  \label{metricbounds}
\ee
 If we have Lebesgue convergence of the metric tensors,
  \be \label{Lp-metric-hyp}
 \int_M |g_j - g_0|_{g_0}^p dV_{g_0} \to 0, \quad p \in [1,\infty],
 \ee
 or volume convergence,
 \be  \label{ConvergenceVolume}
 \vol_{g_j}(M) \to \vol_{g_0}(M),
 \ee
  then the sequence of Riemannian
   manifolds $M_j=(M, g_j)$ converges to $M_0=(M, g_0)$ in both
   the measured Gromov-Hausdorff sense and the volume preserving intrinsic flat sense: 
\be
M_j\mGHto M_0 \textrm{ and } M_j \VFto M_0.
\ee
\end{thm}

In previous work, the authors proved this theorem for sequences of warped Riemannian manifolds  \cite{Allen-Sormani}.   We also presented examples of warped product manifolds demonstrating that there are sequences which converge in each sense that don't converge in the other, and that there are sequences which have different limits depending upon the notion of convergence that is considered.    The convergence theorem was then applied in \cite{AHMPPW1} by the first named author with Hernandez-Vazquez, Parise, Payne, and Wang to prove Gromov's {\em Conjecture on the Stability of the Scalar Torus Rigidity Theorem} in the warped product setting.    The first named author will be applying results from this paper to prove Gromov's Conjecture in the conformal setting.   We believe there should be many other applications of this theorem as well.

In Section \ref{sect: review} we review the definitions of Gromov-Hausdorff (GH) and
metric measure (mGH) convergence, Intrinsic Flat Convergence ($\mathcal{F}$) and
Volume Preserving Intrinsic Flat ($\mathcal{VF}$) convergence and key theorems relating them
in the simplified setting where all the spaces are Riemannian manifolds.
Theorem~\ref{CompactnessThmFlatandGH} specializes a result of Gromov \cite{Gromov-metric}
and Huang-Lee-Sormani \cite{HLS} stating that Riemannian manifolds with bi-Lipschitz bounds
on their distances have subsequences which converge in the uniform, GH, and $\mathcal{F}$ sense 
to the same limit space.   Theorem~\ref{VF-to-mGH-thm} states that if a sequence of
Riemannian manifolds converges in the GH and $\mathcal{F}$ sense to the same Riemannian
limit space, and if the volumes of the manifolds converge to the volume of the limit, then 
the sequences converge in the $\mathcal{VF}$ and mGH sense as well.  With this background the
reader may proceed to read this paper and apply our results without reading any additional articles 
on any of these notions of convergence.

In Section \ref{sect: examples} we include detailed presentations of the eight examples mentioned above.  Some
are conformal to a flat torus and the rest are conformal to a standard round sphere.  The conformal factors are precisely given and the statement of each example is followed by a detailed proof of the claimed convergence and/or lack of convergence for that example.  

In Section \ref{sec:MainThmProof} we give the proof of Theorem \ref{GenGHandFlatConvMetric}. In particular this involves obtaining estimates on volume and control on distances which leads to completing the proof of the main theorems in subsection \ref{subsec:MainProof}. A key new result is given in Theorem \ref{PointwiseAEConvergenceDistances} which shows that a metric lower bound combined with volume convergence implies pointwise convergence of $d_j(p,q) \rightarrow d_0(p,q)$ for almost every $(p,q) \in M\times M$. Due to the uniform bounds on the metric assumed in Theorem \ref{GenGHandFlatConvMetric} we are then able to show uniform, GH, and $\mathcal{F}$ convergence to a length space by applying a theorem of Huang, Lee, and the second named author in the Appendix of \cite{HLS}. By combining with the pointwise almost every convergence of distances we are able to conclude that the length space guaranteed by compactness must be the metric respect to the desired background Riemannian metric.

\textbf{Acknowledgements:} We are grateful to Raquel Perales for closely reading this preprint and suggesting improvements particularly to Theorem 4.4.  We are also grateful to the other participants at the IAS Emerging Topics on Scalar Curvature and Convergence especially Misha Gromov for inspiring conversations.  The second named author would particularly like to thank Alice Chang for finding funding that enabled her to visit IAS and Princeton weekly last year. The first named author would like to thank Ian Adelstein for the invitation to speak at the Filling Volumes, Geodesics, and Intrinsic Flat Convergence conference at Yale University. The first named author would also like to thank Lan-Hsuan Huang and Maree Jaramillo for the invitation to speak at the Spring Eastern Sectional Meeting of the AMS. We would also like to thank the referee for such a careful reading of the paper with such good suggestions for improving the manuscript.

\section{Review} \label{sect: review}

It is our goal that this paper be easily read by Geometric Analysts who are not necessarily experts in the theory of Metric Spaces and Geometric Measure Theory.  We begin by reviewing the fact that $C^0$ bounds on metric tensors
provide Lipschitz bounds on distance functions: observing that $C^0$ Convergence of Riemannian Manifolds
implies Gromov Lipschitz Convergence of the Riemannian manifolds viewed as metric spaces.
We review the
notions of Gromov-Hausdorff (GH) and Intrinsic Flat Convergence ($\mathcal{F}$) in the special setting where we consider only Riemannian
manifolds and not the more singular spaces studied by Metric Geometers.   We then state and review the key compactness theorem we will be applying to prove our results
in this special case where our manifolds are Riemannian.  We next review metric measure ($mGH$) convergence
and Volume Preserving Intrinsic Flat ($\mathcal{VF}$) convergence and a key theorem relating these notions in the Riemannian setting.  Finally we discuss Lebesgue convergence of metric tensors and prior results relating this
notion to $\mathcal{F}$ and GH convergence.   We do not attempt to provide a comprehensive review but focus instead only on the results we need in this paper.   In particular we apologize for not reviewing the extensive literature on conformal convergence.

\subsection{From $C^0$ Convergence of Metric Tensors to Lipschitz Convergence of Distances}
 
Recall that a connected continuous Riemannian manifold $(M,g)$ is a metric space $(M,d)$
with a length
distance 
\be
d_g(p,q)=\inf\{L_g(C): \, C(0)=p,\, C(1)=q, \, \text{piecewise smooth}\}
\ee
where
\be
L_g(C)=\int_0^1 g(C'(s), C'(s)) ^{1/2} \, ds.
\ee

Given bounds on the metric tensor, we have Lipschitz controls on these distances:

\begin{lem}\label{dist-above}
If $g_j$ and $g_0$ are complete continuous Riemannian metrics defined on a connected manifold $M$ so that
\begin{align}
g_j(v,v) \le K g_0(v,v), \quad \forall v \in T_pM 
\end{align} then for $q_1,q_2 \in M$
\begin{align}
 d_j(q_1,q_2) \le K^{1/2} d_{g_0}(q_1,q_2).
\end{align}
\end{lem}

\begin{proof}
Let $\gamma$ be a piecewise smooth curve connecting $q_1, q_2 \in M$. By the assumption that $g_j(v,v) \le K g_0(v,v)$  we can conclude that 
\begin{align}
d_{g_j}(q_1,q_2) &\le \int_{\gamma}(g_j(\gamma'(t),\gamma'(t))^{1/2}\, dt
\\&\le \int_{\gamma}(Kg_0(\gamma'(t), \gamma'(t)))^{1/2}\, dt 
\\&\le K ^{1/2} \int_{\gamma}g_0(\gamma'(t), \gamma'(t))^{1/2}\, dt =K ^{1/2} L_{g_0}(\gamma) .
\end{align}
By taking the infimum over all curves $\gamma$ we find
\begin{align}
d_{g_j}(q_1,q_2) &\le K^{1/2}d_{g_0}(q_1,q_2).
\end{align}

\end{proof}

Thus we see immediately that two-sided bounds on $g_j$ imply bi-Lipschitz bounds on $d_j$:

\begin{lem}\label{biLip-bounds}
If $g_j$ and $g_0$ are complete continuous Riemannian metrics defined on a connected manifold $M$ so that
\be
K_0 g_0(v,v) \le g_j(v,v) \le K_1g_0(v,v) \quad \forall v \in T_pM
\ee
then
\begin{align}
K_0^{1/2} d_0(q_1,q_2) \le  d_j(q_1,q_2) \le K_1^{1/2} d_{g_0}(q_1,q_2).
\end{align}
\end{lem}

This can lead to a lot of confusion when people discuss Lipschitz convergence
without specifying whether they mean Lipschitz convergence of the metric tensor
or Lipschitz convergence of the metric spaces.  Gromov defined Lipschitz convergence 
in \cite{Gromov-metric} of metric spaces, $(M_j,d_j)$ to $(M_0,d_0)$ if there exist
bi-Lipschitz maps from $(M_j,d_j)$ to $(M_0,d_0)$ whose bi-Lipschitz constants converge to $1$.
In other words:
\be
C^0 \textrm{ convergence of } (M,g_j) \to (M,g_0) 
\ee
implies
\be
\textrm{ Gromov Lipschitz convergence of  } (M,d_j) \to (M,d_0).
\ee

A weaker notion of convergence is the uniform convergence of the distance
functions as functions $d_j: M\times M \to [0,D]$ where $D$ is a uniform
upper bound on the diameter of $(M, d_j)$.  We will discuss this uniform convergence
more later.

\subsection{Gromov-Hausdorff Convergence}

The Gromov-Hausdorff (GH) convergence of compact
metric spaces was defined  by Gromov \cite{Gromov-metric} as a way of
weakening the notion of uniform convergence to sequences of distinct metric
spaces.   Although it has been defined for metric spaces it has
proven to be very useful when studying Riemannian manifolds as well.

Recall that a distance preserving map, $\varphi: (M,d) \to (Z, d_Z)$, satisfies
\be
d_Z(\varphi(p), \varphi(p'))=d(p,p') \qquad \forall p,p' \in M.
\ee
In metric geometry books like \cite{Gromov-metric} this is called an ``isometric embedding"
however the notion does not agree with the Riemannian notion of an ``isometric embedding".
For example, the map from the circle to Euclidean space $F: {\mathbb S}^1 \to {\mathbb E}^2$, defined by
\be
F(\theta)=(\cos(\theta), \sin(\theta))
\ee
is a Riemannian isometric embedding but it is not distance preserving because
$
d_{{\mathbb S}^1}(\theta,\theta')
$
is an arclength along the circle while
$
d_{{\mathbb E}^2}(F(\theta'),F(\theta'))$
is the length of a line segment.

The Gromov-Hausdorff (GH) distance between compact metric spaces is then defined
\be
d_{GH}((M_1, d_1), (M_2, d_2)) = \inf d^Z_H(\varphi_1(M_1), \varphi_2(M_2))
\ee
where the infimum is over all compact metric spaces $Z$ and all distance preserving maps
$\varphi_i: M_i\to Z$ and 
$
d^Z_H
$
is the Hausdorff distance between subsets in $Z$:
\be
d_Z^H(A_1, A_2) = \inf\{r:\, \, A_1\subset T_r(A_2)\textrm{ and } A_2 \subset T_r(A_1)\}.
\ee
Here $T_r(A)$ is the tubular neighborhood of radius $r$ about $A$.  So  
\be
A_1\subset T_r(A_2) \iff \forall x_1 \in A_1 \,\,\exists x_2 \in A_2 \textrm{ s.t. } d_Z(x_1,x_2)< r .
\ee

This then defines 
\be
(M_j, d_j)\GHto (M_\infty, d_\infty) \iff d_{GH}((M_j, d_j), (M_\infty, d_\infty)) \to 0.
\ee
This is true iff $\exists$ compact $Z_j$ and distance preserving maps, $\varphi_j: M_j \to Z_j$,
and $\varphi_j': M_\infty \to Z_j$ such that
\be
d_H^{Z_j}(\varphi_j(M_j), \varphi'_j(M_\infty)) \to 0.
\ee

One says that a map,  $\psi: M_1 \to M_2$, is a $\delta$ almost isometry if is it 
$\delta$-almost distance preserving:
\be
|d_2(\psi(p), \psi(p'))-d_1(p,p')| < \delta  \qquad \forall p,p' \in M.
\ee
and is $\delta$ almost onto:
\be
M_2 \subset T_\delta(\psi(M_1)).
\ee
Gromov proved (cf. Corollary 7.3.28 in \cite{BBI} ) that if there is a $\delta$ almost isometry, $\psi: M_1 \to M_2$,
then
\be
d_{GH}(M_1, M_2) < 2\delta.
\ee

In this paper we have a fixed background manifold $M$ for the whole sequence, and we will
see that having this fixed $M$ allows one to define a sequence of common metric spaces $Z_j$
more easily if one has the right hypotheses on the metric tensors.   We will also present some examples
demonstrating that one may not obtain GH convergence when the hypotheses of our main theorem fail
to hold. See the recent work of Aldana, Carron, and Tapie \cite{Aldana-Carron-Tapie} where the authors are able to show GH compactness for sequences of conformal Riemannian manifolds with integral bounds on the scalar curvautre and bounds on the volume.

\subsection{Intrinsic Flat Convergence}

The intrinsic flat distance (${\mathcal{F}}$) defined by Sormani-Wenger in \cite{SW-JDG} is defined for
a large class of metric spaces called integral current spaces.   In their paper they show that ${\mathcal{F}}$ 
is a weaker notion
than Gromov Lipschitz convergence that is distinct from GH convergence and can give different limits.  
In this paper we will explore this further.

As this article is intended for Riemannian geometers, we provide the
definition of ${\mathcal{F}}$ convergence
in the setting where the metric spaces are compact oriented Riemannian manifolds, $(M, g_0)$,
endowed with distance functions, $d_j$, that satisfy
\be\label{d_j-basic}
\lambda \ge \frac{d_j(p,q)}{d_0(p,q)} \ge \frac{1}{\lambda}.
\ee
where $d_0=d_{g_0}$ is the length distance associated to $g_0$ as above.
The distance functions need not arise from Riemannian metrics as long as
they satisfy (\ref{d_j-basic}).  They might for example be the taxi metric on a torus:
\be
({\mathbb T}^2, d_{taxi})= {\mathbb S}^1 \times_{taxi} {\mathbb S}^1
\ee
where
\be
d_{taxi}(p,p')= d_{{\mathbb S}^1}(p_1,p'_1) + d_{{\mathbb S}^1}(p_2, p'_2)
\ee
or other taxi products.

The intrinsic flat (${\mathcal{F}}$) distance is defined similarly to
the Gromov-Hausdorff distance.  Again we are taking an infimum over
all distance preserving maps into a common metric space $Z$ which is now
assumed to be complete instead of compact.  Instead of taking the Hausdorff distance we measure the
Federer-Fleming Flat Distance between the images.  In full generality this flat distance is defined using
Ambrosio-Kirchheim's mass measure of integral currents, $A$ and $B$, lying in $Z$ \cite{SW-JDG} \cite{AK}.   

Here  
we can  take $Z=M\times I$ (where $I$ is an interval )with a well chosen metric $d_Z$
so that $\varphi_i: M_i\to Z$ are distance preserving maps such that
\be
\varphi_0(x)=(x,0) \in Z \quad \forall x \in M_0\textrm{ and } \varphi_1(x)=(x,1)\in Z \quad \forall x\in M_1.
\ee
Since we consider only $M$ without boundary
we can set the filling current, $B$, to be integration over $M\times I$ and estimate its mass
using the Hausdorff measure:
\be
d_{{\mathcal{F}}}((M, d_1), (M, d_0))  \le C_m \mathcal{H}^{m+1}_{d_Z}(M\times I).
\ee
This is an over estimate for the ${\mathcal{F}}$ distance but it suffices to show ${\mathcal{F}}$ convergence 
in this paper.

Note that in general both GH and ${\mathcal{F}}$ convergence are well defined for changing sequences of metric
spaces and have compactness theorems defined in those settings.  It is well known that even if one starts with a 
sequence of oriented Riemannian manifolds $M_j=(M, g_j)$ that the ${\mathcal{F}}$ and GH limits need not even be Riemannian
manifolds.  See \cite{Gromov-metric} and \cite{SW-JDG} for these examples.   It should also be noted that 
GH and ${\mathcal{F}}$ limits need not agree.   There are many examples with
no GH limit, that have ${\mathcal{F}}$ limits \cite{SW-JDG}.

If a sequence has a compact GH limit and the sequence  has a uniform upper bound on volume, 
then a subsequence has a ${\mathcal{F}}$ limit and the
${\mathcal{F}}$ limit is either the zero space or a subset of the GH limit (See Theorem 3.20 in \cite{SW-JDG}). 
Note that if $M_j$ has $\vol_j(M_j)\to 0$ then $M_j$ converges in the intrinsic flat sense to the zero space.  This can also happen without volume converging to $0$ (See the Appendix to \cite{SW-JDG}).   To prove that the limit is not the zero space,
one can examine sequences of balls in $M_j$.  By Lemma 4.1 of \cite{Sormani-ArzAsc} one sees that if $M_j \Fto M_\infty$
and $B(p_j,R) \in M_j$ endowed with the restricted metric $d_j$ converge to a limit $B_\infty(R)$, then for almost every $r\in(0,R)$, $B_\infty(r)$ is isometric to a ball of radius $r$ in $M_\infty$.  In particular if $B(p_j,R)$ do not converge to the zero space then $M_\infty$ is not the zero space either. 

There are also theorems which provide hypotheses proving that ${\mathcal{F}}$ and GH limits exist and agree: see the work of Wenger, Matveev, Portegies, Perales, Nu\~nez-Zimbron, Huang, Lee,
and the second author in \cite{SW-CVPDE} \cite{Portegies-Sormani1} \cite{MatveevPortegies1} \cite{PeralesNunez-Zimbron1} \cite{HLS}.   In the next subsection we describe the key result of Huang, Lee, and the second author that
we apply to prove the main theorems in this paper.

\subsection{Useful Compactness Theorem}\label{subsec:CompactnessThm}

All the results in our paper apply the following compactness theorem to prove intrinsic flat (${\mathcal{F}}$) and Gromov-Hausdorff (GH) convergence of our Riemannian manifolds.  This theorem is an easy consequence of a theorem of Gromov in \cite{Gromov-metric} and a theorem by Huang, Lee, and the third author in \cite{HLS}.   Since both of those earlier theorems are stated in far greater generality than we need here, we simplify things by providing a direct proof here.  Note
that everything is much easier because we assume the spaces are Riemannian manifolds.

\begin{thm}\label{CompactnessThmFlatandGH}
If $M$ is a compact oriented Riemannian manifold with a 
sequence of continuous Riemannian metric tensors, $g_j$, and a background
Riemannian metric, $g_0$, such that 
\begin{align}
\lambda_1 g_0 \le g_j &\le \lambda_2 g_0 ,
\end{align}
 then a subsequence converges in the uniform sense 
 \be
 d_j : M \times M \to {\mathbb R} \quad \rightarrow \quad d_{\infty}: M \times M \to {\mathbb R} 
 \ee
  so that for some $K > 0$
 \begin{align}
 d_0(q_1,q_2) \le d_{\infty}(q_1,q_2) \le K d_0(q_1,q_2).
 \end{align}
 In addition,  for $M_j=(M,g_j)$ and $M_{\infty}=(M, d_\infty)$ we find
\begin{align}
M_j &\Fto M_{\infty}
\\M_j &\GHto M_{\infty}.
\end{align}

\end{thm}

Note that in general uniform convergence of $d_j\to d_\infty$ does not imply $\mathcal{F}$ convergence
due to the possibility of developing a cusp singularity.  This can be seen in Example 3.4 of the first author's paper with Bryden on Sobolev bounds and the convergence of manifolds \cite{Allen-Bryden}.
 
\begin{proof}
First we recall that the Riemannian distances, $d_j$, defined by $g_j$ satisfy the
following bi-Lipschitz bound for some $\lambda$:
\be\label{d_j}
\lambda \ge \frac{d_j(p,q)}{d_0(p,q)} \ge \frac{1}{\lambda}
\ee
(cf. Lemma~\ref{biLip-bounds}).
Note immediately that there is now a uniform upper bound $D$ on the diameter of all the $M_j$,
\be
\diam_{g_j}(M)\le D= \lambda \diam_{g_0}(M).
\ee 
As in Gromov's argument in \cite{Gromov-metric}, applying the Arzela-Ascoli Theorem, a subsequence of the 
\be
d_j: M \times M \to [0,D] \textrm{ converges to } d_\infty: M \times M \to [0,D]
\ee
and it can then be verified that $d_\infty$ is a length metric satisfying  (\ref{d_j}).    So we have
the claimed uniform convergence.  

Gromov then explains in \cite{Gromov-metric} how this uniform convergence implies GH convergence.
In fact Gromov-Hausdorff convergence is an extension of the notion of uniform convergence.   We do not
apply his proof.

In the Appendix to \cite{HLS}, Huang, Lee, and the second author prove both GH and ${\mathcal{F}}$ convergence
by constructing a common metric space 
\be
Z_j= [-\varepsilon_j, \varepsilon_j] \times M.
\ee
where
\be\label{epsj}
\epsilon_j= \sup\left\{|d_j(p,q)-d_\infty(p,q)|:\,\, p,q\in M\right\} 
\ee
with a metric $d'_j$ on $Z_j$ created by
gluing together a pair of taxi products
\be
(Z_j^-, d'_j)= [- \varepsilon_j, 0] \times_{taxi} (M, d_j)
\ee
and
\be
(Z_j^+, d'_j)= [0, \varepsilon_j] \times_{taxi} (M, d_\infty).
\ee
In the lemma in the Appendix of \cite{HLS} they prove
there are distance preserving maps 
$\varphi_j:(M, d_j)\to (Z_j, d'_j)$ and
$\varphi_j':(M, d_\infty)\to (Z_j, d'_j)$ such that
\be
\varphi_j(p)=(-\varepsilon_j, p)
\textrm{ and } \varphi'_j(p)=(\varepsilon_j, p).
\ee
They also show
\be\label{d0'}
d'_j(z_1,z_2) \le d'_0((t_1,p_1),(t_2, p_2)):= |t_1-t_2|+ \lambda d_0(p_1, p_2).
\ee

Huang, Lee, and the third author then apply (\ref{d0'}) 
to observe that every point 
\be
(p, -\varepsilon_j)=\varphi_j(p) \in \varphi_j(M_j)
\ee
has a point
\be
(p, \varepsilon) =\varphi_j'(p)\in \varphi_j'(M_\infty)
\ee
such that
\be
d_j'((p, -\varepsilon_j), (p, \varepsilon) )= 2\varepsilon_j
\ee
and visa versa.   Thus
\be
d_{GH}(M_j, M_\infty) \le d_H^{Z_j}(\varphi_j(M_j), \varphi_j'(M_\infty)) = 2 \varepsilon_j \to 0.
\ee
To estimate the intrinsic flat distance one then needs only estimate
\begin{eqnarray}
d_{{\mathcal{F}}}(M_j, M_\infty) &\le& C_m \mathcal{H}^{m+1}_{d'_0}(Z_j) \\
&\le& C_m 2\varepsilon_j  \lambda^m \mathcal{H}^{m}_{d_0}(M)
\end{eqnarray}
More details on this with more precise constants appear in the Appendix to \cite{HLS}.
\end{proof}

\subsection{Metric Measure Convergence and Volume Preserving Intrinsic Flat Convergence}

The notion of metric measure convergence first introduced by Fukaya in \cite{Fukaya-mm},
and studied by Cheeger-Colding in \cite{Cheeger-Colding-1}, and generalized by Sturm to an
intrinsic Wasserstein distance in \cite{Sturm-mm-1,Sturm-mm-2} is defined on a large class of spaces as well.  If one
has a sequence of compact oriented Riemannian manifolds converging in the GH sense such that whenever points
$p_j \in M_j$ converge to $p_\infty\in M_\infty$ the volumes of the balls around them converge
to the measure of the limit ball,
\be
\vol(B(p_j,r)) \to \mu(B(p_\infty,r))
\ee 
then one has metric measure convergence in all these respects to $(M_\infty, d_\infty, \mu)$.
One might ask how $p_j \in M_j$ converge to $p_\infty\in M_\infty$ when they do not lie in a common
space.  This is said to hold when there are distance preserving maps $\varphi_j: M_j \to Z$ and
$\varphi_\infty: M_\infty \to Z$ such that
\be\label{mm-Fukaya}
d_H^{Z}(\varphi_j(M_j), \varphi_\infty(M_\infty)) \to 0 \textrm{ and } \varphi_j(p_j)\to \varphi_\infty(p_\infty).
\ee
The existence of such a common compact $Z$ was proven in general by Gromov in \cite{Gromov-poly} whenever a
sequence converges in the GH sense to a compact limit.    The existence of a common complete $Z$
was proven in \cite{SW-JDG} for intrinsic flat converging sequences and studied further in \cite{Sormani-ArzAsc}.

A common $Z$ that works for both ${\mathcal{F}}$ and GH convergence in the setting of Theorem~\ref{CompactnessThmFlatandGH} was constructed by the authors in the
Appendix to \cite{Allen-Bryden} by attaching the many $Z_j$ like pages along the
$M_\infty$ edge.  There we saw that if we have $(M, g_j)$ and take $p_j=p$ a fixed point in $M$
viewed in each $M_j$ then $p_\infty=p$ as well now viewed in $M_\infty$.  Thus one can simplify the
definition in our setting to say that
$M_j=(M,g_j)$ converges in the metric measure sense to $(M, d_\infty)$ if it converges in the Gromov-Hausdorff
sense and 
\be\label{mm-here}
\vol_j(B(p,r))\to \mathcal{H}^m_{d_\infty}(B(p,r)) \qquad \forall p \in M \qquad \forall r>0.
\ee

Volume preserving intrinsic flat ($\mathcal{VF}$) convergence is defined 
\be \label{VF-defn}
M_j \VFto M_\infty \iff M_j\Fto M_\infty \textrm{ AND } \vol_{g_j}(M_j)\to \vol_{g_\infty}(M_\infty).
\ee
when the sequence and the limit are both Riemannian manifolds.  This has been studied
by Jauregui and Lee in \cite{Jauregui-Lee-SWIF}.  It should be noted that the second author and Wenger
proved in \cite{SW-JDG} (see also \cite{Sormani-ArzAsc} \cite{Portegies-Sormani1}) that if
$M_j \Fto M_\infty$ and $p_j \to p_\infty$ then for almost every $r>0$
$B(p_j,r) \Fto B(p_\infty,r)$ and 
\be
\liminf_{j\to \infty} \vol_{g_j}(B(p_j, r) ) \ge \vol_{g_\infty}(B(p_\infty,r)).
\ee
If one requires $\vol_{g_j}(M_j)\to \vol_{g_\infty}(M_\infty)$ then no balls can drop in volume in the limit.
Thus
\be
(M_j, g_j) \VFto (M_\infty, g_\infty) \implies \vol_{g_j}(B(p_j, r) ) \to \vol_{g_\infty}(B(p_\infty,r)).
\ee
In particular if we also have $(M_j,g_j) \GHto (M_\infty, g_\infty)$ then we have
$(M_j,g_j) \mGHto (M_\infty, g_\infty)$.

We summarize this in a theorem which should be viewed as a simplification of known theorems:

\begin{thm} \label{VF-to-mGH-thm}
Let $M$ be a compact oriented manifold and $M_j=(M,g_j)$ and $M_{\infty}=(M, g_\infty)$ be Riemannian manifolds. If $M_j \GHto M_{\infty}$, $M_j \Fto M_\infty$, and
\be
\vol_{g_j}(M_j)\to \vol_{g_\infty}(M_\infty)
\ee
then $M_j \VFto M_\infty$ and $M_j \mGHto M_\infty$.
\end{thm}

\subsection{Lebesgue Convergence of Riemannian Manifolds}

Above we have reviewed four geometric notions of convergence that are all weaker than $C^0$ convergence of
metric tensors.  However geometric analysts will find it more natural to weaken the notion of $C^0$ convergence to
$L^p$ convergence of metric tensors.   Indeed when studying sequences of Riemannian manifolds it can be easier 
for geometric analysts to prove $L^p$ convergence rather than GH or $\mathcal{F}$ convergence. Our results here can be thought of as analogous to how in PDEs one will start by obtaining $L^p$ bounds in order to use this control to then bootstrap up to stronger control in the future.

Here we review the definition of $L^p$ norm for Riemannian metrics and discuss important properties for this paper. See the work of Clarke \cite{Clarke} for these definitions used for a study of the space of Riemannian metrics on a fixed background manifold with respect to the $L^2$ topology.

If we consider the compact manifold $M$ and the Riemannian metrics $(M,g_0)$ and $(M,g_1)$ then we can define the $L^p$, $p \ge 1$ norm of $g_1$ with respect to the background metric $g_0$ to be
\begin{align}
    \|g_1\|_{L^p_{g_0}(M)} = \lp \int_M |g_1|_{g_0}^p dV_{g_0} \rp^{1/p}.
\end{align}
Notice that this is just the usual definition of the $L^p$ norm on $(M,g_0)$ for the function $|g_1|_{g_0}$
which is defined by letting $\lambda_1^2,...,\lambda_m^2$ be the eigenvalues of $g_1$ with respect to $g_0$ with corresponding eigenvectors $v_1,...,v_m$ so that
\begin{align}
    g_1(v_i,v_i) = \lambda_i^2 g_0(v_i,v_i)=\lambda_i^2, \quad 1 \le i \le m,
\end{align}
and 
\be
|g_1|_{g_0}= \sqrt{\sum_{i=1}^m \lambda_i^4} ,
\ee
the norm of $g_1$ with respect to $g_0$.
In fact, if $g_1 = f_1^2g_0$ is a conformal metric then
\begin{align}
   \|g_1\|_{L^{\frac{p}{2}}_{g_0}(M)} = m^{\frac{1}{2}} \|f_1\|_{L^p_{g_0}(M)}^2 = m^{\frac{1}{2}}\lp \int_M |f_1|^p dV_{g_0} \rp^{2/p}.
\end{align}

Hence we say that a sequence of Riemannian manifolds $(M,g_j)$ converges to the Riemannian manifold $(M,g_{\infty})$ in $L^p$ norm with respect to $g_0$ if
\begin{align}
    \|g_j - g_{\infty}\|_{L^p_{g_0}(M)} \rightarrow 0.
\end{align}

The following lemma follows from H\"older's Inequality:

\begin{lem} Let $g_j, g_0$ be continuous Riemannian metrics defined on the compact manifold $M$. For $p \in (1,\infty]$ if
\begin{align}
    \|g_j-g_0\|_{L_{g_0}^p(M)}\rightarrow 0, 
\end{align}
then
 \begin{align}
     \|g_j-g_0\|_{L_{g_0}^q(M)} \rightarrow 0, \quad  1 \le q \le p.
  \end{align}
  \end{lem}
  
In general the careful treatment of the power of $p$ in $L^p$ convergence is crucial. 
However when we assume a uniform upper bound on the metric tensors as in our main theorem,
the specific value of $p$ chosen in $L^p$ convergence is not particularly important:

\begin{lem}\label{LpToLq}
If $(M,g_j)$, $(M,g_0)$ are compact continuous Riemannian manifolds so that
\begin{align}
     g_j(v,v) \le K g_0(v,v), \quad \forall v \in T_pM
\end{align}
and for $p \in [1,\infty]$
\begin{align}
    \|g_j-g_0\|_{L_{g_0}^p(M)}\rightarrow 0, 
\end{align}
then
\begin{align}
    \|g_j-g_0\|_{L_{g_0}^q(M)}, \forall q \in [1,\infty].
\end{align}
\end{lem}

\begin{proof}
We know by H\"older's inequality that 
  \begin{align}
     \|g_j-g_0\|_{L_{g_0}^q(M)} \rightarrow 0, \quad  1 \le q \le p.
  \end{align}
  Otherwise, we calculate for $q > p$
  \begin{align}
     \|g_j-g_0\|_{L_{g_0}^q(M)}^q &= \int_M|g_j-g_0|^q dV_{g_0}
     \\&= \int_M|g_j-g_0|^p |g_j-g_0|^{q-p}dV_{g_0}
     \\ &\le (2K\sqrt{m})^{q-p} \int_M|g_j-g_0|^p dV_{g_0} \rightarrow 0.
  \end{align}
\end{proof}

We now state a standard analysis result relating convergence in $L^p$ to convergence in $L^p$ norm.

\begin{lem}\label{LpToLpNorm}
Assume that $M_j^m = (M,g_j)$, $M_{\infty}=(M,g_{\infty})$, and $M_0=(M,g_0)$ are continuous Riemannian manifolds. If for $p \in [1,\infty)$
\begin{align}
    \int_M|g_j-g_{\infty}|_{g_0}^pdV_{g_0} \rightarrow 0
\end{align}
then
\begin{align}
    \|g_j\|_{L_{g_0}^p(M)} \rightarrow \|g_{\infty}\|_{L_{g_0}^p(M)}.
\end{align}
\end{lem}
\begin{proof}
By the reverse triangle inequality applied to norms
\begin{align}
    \left|\|g_j\|_{L_{g_0}^p(M)} - \|g_{\infty}\|_{L_{g_0}^p(M)} \right| \le \|g_j-g_{\infty}\|_{L_{g_0}^p(M)},
\end{align}
we find that $g_j$ converging to $g_{\infty}$  in $L_{g_0}^p(M)$ implies convergence in $L_{g_0}^p(M)$ norm.
\end{proof}

In a previous paper by the authors \cite{Allen-Sormani} a comprehensive comparison of $L^p$ convergence and the uniform, GH, and $\mathcal{F}$ convergence of warped products Riemannian manifolds was given. These are Riemannian
manifolds whose metric tensor has the form $g=dt^2 +f^2(t) g_0$.  The authors gave a theorem which assumed metric bounds and $L^2$ convergence of warping factors, $f_j \rightarrow f_{0}$ which implied uniform, GH, and $\mathcal{F}$ convergence to a warped product with warping function $f_{0}$.  That theorem is now a special case of the
the Theorem~\ref{GenGHandFlatConvMetric} proven here.   

The authors also produced many examples contrasting different notions of convergence in  \cite{Allen-Sormani}.  In
particular they constructed an examples  which show that $L^p$ convergence of Riemannian manifolds need not
agree with geometric notions of convergence like GH and $\mathcal{F}$ convergence. This is due to the fact that the $L^p$ norm considers $g_j$ and $g_{0}$ to be close even if they measure the lengths of vectors very differently on a set of measure zero $S \subset M$.  In the examples where the notions of convergence disagree, the
set $S$ contains geodesics with respect to $g_{\infty}$.   It is particularly a concern if the $g_j$ are much smaller than
$g_0$ on the set $S$ providing shortcuts for the $d_j$ so that the limit of the $d_j$ ends up smaller than $d_\infty$.
There are pictures in  \cite{Allen-Sormani} illustrating exactly what is happening.  To avoid this trouble the 
hypothesis that $g_j \ge (1-1/j)g_0$ was first introduced in \cite{Allen-Sormani} rather than simply
$g_j\ge (\lambda_1) g_0$ as is needed to apply the useful compactness theorem (cf. Theorem~\ref{CompactnessThmFlatandGH}).   Example~\ref{Cinched-Sphere} within this paper demonstrates
that the same issue arises when studying conformal sequences of manifolds and so we also require this hypothesis
in our new Theorem~\ref{GenGHandFlatConvMetric}.

\section{Examples of Sequences of Conformal Manifolds}\label{sect: examples}

In this section we explore sequences of metrics on a torus that are conformal to the flat torus and metrics on a sphere that are conformal to the standard sphere.  The first three examples are all uniformly bounded and directly explore hypotheses and conclusions related to Theorem \ref{ConfGHandFlatConv}. The last five examples explore what can happen when the uniform upper bound of Theorem \ref{ConfGHandFlatConv} is removed.

We also direct the reader to the warped product examples given in our previous paper \cite{Allen-Sormani} which are relevant to Theorem \ref{GenGHandFlatConvMetric}.    

\subsection{Without Lower Bounds on the Conformal Factor}

Here we see an example which shows that without the $C^0$ bound from below one cannot prove convergence of $M_j$ to $M_0$ in Theorem \ref{ConfGHandFlatConv}.   Note examples 
demonstrating why the lower bounds are needed on the metric tensor itself 
as in Theorem \ref{GenGHandFlatConvMetric} appeared in the authors' previous paper
on warped products \cite{Allen-Sormani}. 

 \begin{ex} \label{Cinched-Sphere}  
 Consider the sequence of functions on $\Sp^m$ which are radially defined from the north pole  
 \be
 f_j(r)=
 \begin{cases}
 1 & r\in[0,\pi/2- 1/j]
 \\  h(jr-\pi/2) & r\in[\pi/2- 1/j, \pi/2+ 1/j]
 \\ 1 &r\in [\pi/2+ 1/j, \pi]
 \end{cases}
\ee
where $h:[-1,1]\rightarrow \R$ is a smooth even function such that 
$h(-1)=1$ with $h'(-1)=0$, 
decreasing to $h(0)=h_0\in (0,1)$ and then
increasing back up to $h(1)=1$, $h'(1)=0$. We will see that for $M_j = (\Sp^m, f_j^2 g_{\Sp^m})$
\begin{align}
M_j &\VFto M_{\infty}
\\M_j &\mGHto M_{\infty}
\end{align}
but we can conclude that $M_{\infty}$ is not isometric to $\Sp^m$. Instead $M_{\infty} = (\Sp^m, f_{\infty}^2 g_{\Sp^m}) $ is the conformal metric with conformal factor
 \be
 f_{\infty}(r)=
 \begin{cases}
 h_0 & r=\pi/2
 \\  1 &\text{ otherwise}
 \end{cases}.
 \ee
 The distances between pairs of points near the equator in this limit space is achieved by geodesics which run to the
 equator, and then around inside the cinched equator, and then out again.  
\end{ex}

\begin{proof}
Notice that $h_0 \le f_j \le 1$ and hence by Theorem \ref{CompactnessThmFlatandGH} we find that on a subsequence
\begin{align}
M_j &\Fto M',
\\M_j &\GHto M',
\end{align}
for some compact metric space $M'=(\Sp^m,d')$ where we have uniform convergence $d_j \rightarrow d'$.
Now our goal is to show that $M'=M_{\infty}$ by showing pointwise convergence $d_j(q_1,q_2) \rightarrow d_{\infty}(q_1,q_2)$ for all $q_1,q_2 \in \Sp^m$.   Thus $d'=d_\infty$.

Let $q_1,q_2 \in \Sp^m$ and consider $\gamma(t) = (r(t),\theta(t))$ to be any curve in $\Sp^m$. We can compute
\begin{align}
d_j(q_1,q_2) \le L_j(\gamma) =  \int_0^{L} f_j(r(t))\sqrt{r'(t)^2+\theta'(t)^2} dt \rightarrow L_{\infty}(\gamma)
\end{align}
where the convergence follows from the dominated convergence theorem since $f_j(r(t)) \rightarrow f_{\infty}(r(t))$ pointwise and $f_j$ is uniformly bounded. Since this is true for any curve $\gamma$ we see that 
\begin{align}
\limsup_{j \rightarrow \infty} d_j(q_1,q_2)  \le d_{\infty}(q_1,q_2)\label{limsupConv}.
\end{align}

Now let $\gamma_j(t) = (r_j(t),\theta_j(t))$ be the length minimizing curve with respect to $g_j$ defined on $t \in [0,L]$. Then if we define 
\begin{align}
    S_j = \{t \in [0,L] : \pi/2-1/j \le r_j(t) \le \pi/2 + 1/j, r_j(t) \not=\frac{\pi}{2}\},
\end{align} 
\begin{align}
    T_j =\{t \in [0,L] : r_j(t) =\frac{\pi}{2}\},
\end{align}
and 
\begin{align}
    U_j  = [0,1] \setminus (S_j \cup T_j)
\end{align} 
we can compute
\begin{align}
d_j(q_1,q_1) &= L_j(\gamma_j)
\\& =L_j(U_j) + L_j(T_j) +L_j(S_j)
\\&= L_{\infty}(U_j) + h_0 L_{\Sp^m}(T_j)+ L_j(S_j) 
\\& \ge d_{\infty}(q_1,q_2) +L_j(S_j) -  L_{\Sp^m}(S_j).
\end{align}
Now we notice
\begin{align}
|L_j(S_j) -  L_{\Sp^m}(S_j)|&\le \int_0^{L_{\Sp^m}(S_j)}
|h_j(r_j(t),\theta_j(t)) - 1|  dt \rightarrow 0
\end{align}
since $L_{\Sp^m}(S_j) \le L < \infty$, otherwise we would contradict \eqref{limsupConv}, and the dominated convergence theorem. Hence we find
\begin{align}
\liminf_{j \rightarrow \infty} d_j(q_1,q_2)  \ge d_{\infty}(q_1,q_2)\label{liminfConv}.
\end{align}
By combining \eqref{limsupConv} and \eqref{liminfConv} we find pointwise convergence of distances which implies that $(M',d') = (M_{\infty}, d_\infty)$. It is also clear that the volume converges in this example.
\end{proof}

\subsection{Conformal Tori converging as in Theorem\ref{ConfGHandFlatConv} to a Flat Torus}

Here we give an example which fits the hypotheses of the Theorem \ref{GenGHandFlatConvMetric} and comes to the same conclusion but whose conformal factors do not converge in $C^0$. This shows that one should not expect to conclude any stronger convergence in Theorem \ref{GenGHandFlatConvMetric}. 

\begin{ex}\label{No C0 Convergence}
Consider the sequence of functions on $\Tor^m$ which are radially defined from a point $p \in \Tor^m$
\begin{equation}
f_j(r)=
\begin{cases}
K &\text{ if } r \in [0,1/j]
\\h_j(jr) &\text{ if } r \in [1/j,2/j]
\\ 1 & \text{ if } r \in (1/j,\sqrt{2}\pi]
\end{cases}
\end{equation}
where $K \in (1,\infty)$ and $h_j:[1,2] \rightarrow \R$ is a smooth, decreasing function so that $h_j(1) = K$ and $h_j(2) = 1$. Then $M_j = (\Tor^m, f_j^2 g_{\Tor^m})$ converges to $\Tor^m$ in $L^p$ $\forall p \in (0,\infty)$, but does not converge to $\Tor^m$ in $C^0$, and rather
\begin{align}
M_j &\VFto \Tor^m
\\M_j &\GHto \Tor^m.
\end{align}
\end{ex}
\begin{proof}
We begin by computing
\begin{align}
\vol(M_j) &= \int_{\Tor^m}f_j^m dVol 
\\&=  Vol(B^m(p,1/j))K 
\\&\qquad+ \omega_m \int_{1/j}^{2/j} h_j(jr)^m r^{m-1} dr
\\&\qquad + Vol(\Tor^m \setminus B^m(p,2/j)\rightarrow Vol(\Tor^m),
\end{align}
\begin{align}
\diam(M_j) &= \int_0^{\sqrt{2}\pi}f_j dr \label{DiamContEx3.1} 
\\&= \frac{K}{j}+\int_{1/j}^{2/j} h_j(jr) dr + (\sqrt{2}\pi - 1/j) \rightarrow \sqrt{2}\pi,
\end{align}
and
\begin{align}
\|f_j-1\|_{L^p(\Tor^m)} &\le Vol(B(p,2/j)) (K-1)^p  
\rightarrow 0.
\end{align}

Hence $M_j \mGHto \Tor^m$ and  $M_j \VFto \Tor^m$  by Theorem \ref{ConfGHandFlatConv} but we note that $f_j$ clearly does not converge uniformly to $1$
\end{proof}

\subsection{Only Pointwise Convergence of the Conformal Factor}

Here we give an example of a sequence of conformally flat tori whose conformal factors pointwise converge on a dense set to $1$, but do not converge in $L^p$ to $1$.  The conformal factors do satisfy a lower bound $f_j \ge 1$
and a uniform upper bound, but $M_j$ do not GH nor ${\mathcal{F}}$ converge to $M_0$. This shows that the assumption of $L^p$ convergence or volume convergence in Theorem \ref{GenGHandFlatConvMetric} cannot be weakened to pointwise convergence on a dense set. 

\begin{ex}\label{PointwiseOnDenseSet}
Define the $j$th lattice
\begin{align}
    L_j= \left\{ \lp\frac{n}{2^j},r\rp \cup \lp r, \frac{m}{2^j}\rp: r \in [0,1], n,m = 1...2^j\right\},
\end{align}
and the limiting set
\begin{align}
    L = \bigcup_{j=1}^{\infty} L_j.
\end{align}
Consider $M_j=(\Tor^m,g_j=f_j^2g_0)$ with conformal factor,
\begin{align}
f_j=
    \begin{cases}
    1 & \text{ on } T_{\frac{1}{2^{j+2}}}\lp L_j\rp
    \\ \sqrt{2} &\text{ elsewhere}
    \end{cases},
\end{align}
where $T_r(S)$ represents a tubular neighborhood of radius $r>0$ around the set $S\subset \Tor^m$. 
Then
\begin{align}
    f_j(p) \rightarrow 1, \qquad p \in L
\end{align} 
pointwise on the dense set $L$,

\begin{align}
    f_j(p) \rightarrow \sqrt{2}, \qquad p \in \Tor^m\setminus L
\end{align} 
pointwise on the full measure set set $\Tor^m\setminus L$,
and 
\begin{align}
g_j \ge g_{\Tor^m},
\end{align}
but  yet $M_j$ does not converge to $\Tor^m$. Instead 
\begin{align}
    M_j \GHto \Tor^m_{taxi},
    \\ M_j \Fto \Tor^m_{taxi},
\end{align}
where  $\Tor^m_{taxi} = (\Tor^m, d_{taxi})$ is a torus with a taxi metric.
\end{ex}

\begin{proof}
First we notice that for any $p,q \in M_j$ the length minimizing geodesic from $p$ to $q$ consists of a collection of straight line segments. Furthermore, no segment should enter a region where $f_j = \sqrt{2}$, unless it begins or ends in one, since 
\begin{align}\label{TaxiObservation}
    \sqrt{2} \sqrt{x^2+y^2} \ge |x|+|y|.
\end{align}
By this observation we notice that for all $p \in M_j$ there exists $\lp \frac{n}{2^j},\frac{m}{2^j} \rp$ so that
\begin{align}
    d_j\lp p,\lp \frac{n}{2^j},\frac{m}{2^j} \rp \rp \le \sqrt{2} \lp\frac{\sqrt{2}}{2^j} \rp,\label{NetEstimate}
\end{align}
which is an overestimate of the distance from the center of the square to one of its corners. 
So \eqref{NetEstimate} shows that $L_j$ is a $\frac{1}{2^{j-1}}$ net and since by \eqref{TaxiObservation}
\begin{align}
    d_j(p,q) \rightarrow d_{taxi}(p,q) \qquad \forall p,q \in L
\end{align}
 we have that 
\begin{align}
    M_j \GHto \Tor^m_{taxi}.
\end{align} 
Then since
\begin{align}
    g_0 \le g_j \le \sqrt{2} g_0,
\end{align}
we have by Theorem \ref{CompactnessThmFlatandGH} that the ${\mathcal{F}}$ limit agrees with the GH limit.
\end{proof}

\subsection{Conformal Tori converging in $L^p$, $p < \frac{m}{\alpha}$, $\alpha \in (0,1)$}

Here we see an example which satisfies $g_j \ge g_0$ and whose conformal factor $f_j$ converges in $L^p$ norm for all $p < \frac{m}{\alpha}$, $\alpha \in (0,1)$ which is a stronger assumption than volume convergence, and hence we see that $M_j$ converges in both the mGH and $\mathcal{VF}$ sense to a flat torus. 

\begin{ex}\label{L^p Conv}
Consider the sequence of functions on $\Tor^m$ which are radially defined from a point $p \in \Tor^m$
\begin{equation}
f_j(r)=
\begin{cases}
j^{\alpha} &\text{ if } r \in [0,1/j]
\\h_j(jr) &\text{ if } r \in [1/j,2/j]
\\ 1 & \text{ if } r \in (1/j,\sqrt{2}\pi]
\end{cases}
\end{equation}
where $0< \alpha < 1$ and $h_j:[1,2] \rightarrow \R$ is a smooth, decreasing function so that $h_j(1) = j^{\alpha}$ and $h_j(2) = 1$. Then for $M_j = (\Tor^m, f_j^2 g_{\Tor^m})$ 
\begin{align}
    \|f_j-1\|_{L^p(\Tor^m)} \rightarrow 0 \text{ for }p < \frac{m}{\alpha},
\end{align}
 \begin{align}
    1 \le f_j,
 \end{align}
and
\begin{align}
M_j &\VFto \Tor^m
\\M_j &\mGHto \Tor^m.
\end{align}
\end{ex}
\begin{proof}
We begin by computing
\begin{align}
\vol(M_j) &= \int_{\Tor^m}f_j^m dVol 
\\&\le  Vol(B^m(p,2/j))j^{m \alpha} 
\\&\qquad+ \omega_m \int_{1/j}^{2/j} h_j(jr)^m r^{m-1} dr
\\&\qquad + Vol(\Tor^m \setminus B^m(p,2/j)\rightarrow Vol(\Tor^m),
\end{align}
\begin{align}
\diam(M_j) &= \int_0^{\sqrt{2}\pi}f_j dr \label{DiamContEx3.1} 
\\&= j^{\alpha -1}+\int_{1/j}^{2/j} h_j(jr) dr + (\sqrt{2}\pi - 1/j) \rightarrow \sqrt{2}\pi,
\end{align}
and
\begin{align}
\|f_j-1\|_{L^p(\Tor^m)} &\le Vol(B(p,2/j)) (j^{\alpha}-1)^p  
\\&= 2^m\omega_m \frac{(j^{\alpha}-1)^p}{j^m} 
\\&\le 2^m\omega_m j^{\alpha p-m} \rightarrow 0, \quad p < \frac{m}{\alpha}.
\end{align}

Now by Lemma \ref{biLip-bounds} since $f_j \ge 1$ we know
\begin{align}\label{DistLowerBound}
    d_j(q_1,q_2) \ge d_0(q_1,q_2) \quad \forall q_1,q_2 \in \Tor^m.
\end{align}

For $q_1,q_2 \in \Tor^m$ so that the $g_0$ length minimizing geodesic, $\gamma$, does not pass through $p$ we notice that for $j$ chosen large enough $L_j(\gamma) = L_0(\gamma)$ and hence
\begin{align}\label{FirstDistBound}
    d_j(q_1,q_2) \le L_j(\gamma) = L_0(\gamma) = d_0(p,q).
\end{align}
If $\gamma$ does pass through $p$ then it is enough to consider $p,q\in \Tor^m$ and for $j$ chosen large enough we have that $q \not \in B(p,2/j)$ and hence we can calculate
\begin{align}\label{SecondDistBound}
    d_j(p,q) &\le \int_0^{d_0(p,q)}f_j dr= 2j^{\alpha-1} + d_0(p,q)-\frac{2}{j} = d_0(p,q) + 2\frac{j^\alpha-1}{j}.
\end{align}
Now by combining \eqref{DistLowerBound},\eqref{FirstDistBound}, and \eqref{SecondDistBound} we find $d_j \rightarrow d_0$ uniformly
 and hence 
 \be
 M_j \mGHto \Tor^m.
 \ee 
 Once $M_j$ has a GH limit $M_\infty$ then by Theorem 3.20, p. 147 of \cite{SW-JDG} we know that a subsequence converges
in the $\mathcal{F}$ sense either to the $0$ space or a subset of the GH limit.   Since the $\mathcal{F}$
limit must have the same dimension as the sequence, the subsequence converges
in the $\mathcal{F}$ sense either to the $0$ space or a subset of the flat torus.   Now consider any point $q\neq p$ in
the flat torus and consider $B_{d_j}(q, r) \in M_j$ where $r<d_0(p,q)\le d_j(p,q)$.   For $j$ sufficiently large this
is a Euclidean ball, and so by Lemma 4.1 in \cite{Sormani-ArzAsc}, choosing possibly a smaller $r>0$, $B_{d_j}(q, r) \in M_j$ converges to a Euclidean ball of radius $r$ inside
the $\mathcal{F}$ limit, $M_\infty$.  Piecing together these Euclidean balls, we see that the intrinsic flat limit is a flat torus
with possibly one point removed.  But that point is added back in when one takes the metric completion and hence we find $\mathcal{VF}$ convergence.  A new proof of this limit will appear in upcoming work of the authors.

\end{proof}

\subsection{Conformal Tori which don't converge in $L^m$}

Here we see an example of a sequence of metrics on tori whose conformal factor $f_j$ is not bounded in $L^p$ for any $p > m$, is bounded in $L^m$ but does not converge in $L^m$, and does not satisfy volume convergence. This example will converge in both the GH and $\mathcal{F}$ sense to a metric space which is isometric to a flat torus with a flat disk attached so that the boundary of the disk is identified with a point in the flat torus. This shows that volume convergence is important in order to prevent bubbling. From an analytic point of view this sequence of conformal factors are converging to a measure whose mass is concentrating at a point which geometrically corresponds to a flat torus with a bubble attached at a point.

\begin{ex}\label{No L^m Conv}
Consider the sequence of functions on $\Tor^m$ which are radially defined from a point $p \in \Tor^m$
\begin{equation}
f_j(r)=
\begin{cases}
j &\text{ if } r \in [0,1/j]
\\h_j(jr) &\text{ if } r \in [1/j,2/j]
\\ 1 & \text{ if } r \in (2/j,\sqrt{m}\pi].
\end{cases}
\end{equation}
where $h_j:[1,2] \rightarrow \R$ is a smooth, decreasing function so that $h_j(1) = j$, $h_j'(1)=h_j'(2)=0$, and $h_j(2) = 1$ so that
\begin{align}
    \frac{1}{j^m}\int_1^2h_j(s)^m s^{m-1}ds \rightarrow 0.\label{ConstructionHyp}
\end{align}
Then $f_j$ is not bounded in $L^p$ norm for $p > m$ but does have bounded $L^m$ norm and volume.
Furthermore, for $M_j=(\Tor^m, f_j^2 g_{\Tor^m})$
\begin{align}
M_j &\Fto M_{\infty}
\\M_j &\GHto M_{\infty}
\end{align}
where $M_{\infty}$ is not isometric to $\Tor^m$. Instead 
\begin{align}
    M_{\infty} = \Tor^m \sqcup \mathbb{D}^m/\sim,
\end{align}
where we fix a $p \in \Tor^m$ and for $d \in \partial \mathbb{D}^m$  we have
\begin{align}
    p \sim d.
\end{align}
\end{ex}

\begin{proof}
First notice by Holder's inequality that the assumption \eqref{ConstructionHyp} implies
\begin{align}
    \frac{1}{j} \int_1^2 h_j(r)dr & \le \frac{1}{2j} \int_1^2 h_j(r)r^{\frac{m-1}{m}}dr
    \\&\le \frac{1}{2j} \lp \int_1^2 h_j(r)^m r^{m-1}dr \rp^{1/m} 
    \\&= \frac{1}{2}\lp \frac{1}{j^m}\int_1^2 h_j(r)^m r^{m-1}dr \rp^{1/m} \rightarrow 0.
\end{align}
Now we begin by computing
\begin{align}
\vol(M_j) &= \int_{\Tor^m}f_j^m dVol 
\\&= Vol(B^m(p,1/j))j^m + Vol(\Tor^m \setminus B^m(p,2/j))
\\&\qquad+ \omega_m \int_{1/j}^{2/j} h_j(jr)^m r^{m-1}dr
\\&\rightarrow Vol(B^m(p,1))+Vol(\Tor^m),
\end{align}
\begin{align}
    \|f_j\|_{L^p} \ge Vol(B^m(p,1/j))j^p = \omega_m j^{p-m} \rightarrow \infty, \quad p > m,
\end{align}
\begin{align}
\diam(M_j) &= \int_0^{\sqrt{m}\pi}f_j dr 
\\&= 1 +\int_{1/j}^{2/j} h_j(jr) dr+ (\sqrt{m} \pi  - 2/j) 
\\& \rightarrow 1+ \sqrt{m}\pi.
\end{align}
So we see that the volume and diameter do not converge to the volume or diameter of $M_0$ but we do note that $g_j \ge g_0$ by construction.

Now we would like to show that $M_j$ converges to $M_{\infty}$ in the GH and ${\mathcal{F}}$ sense. We begin by constructing a sequence of maps in order to estimate the GH distance
\begin{align}
    F_j:M_{\infty}\rightarrow M_j
\end{align}
where
\begin{align}
    F_j(\mathbb{D}^m\setminus \partial \mathbb{D}^m) &= B_{\Tor^m}(p,1/j),
    \end{align}
    is defined by scaling,
    \begin{align}
  F_j(\Tor^m \setminus B_{\Tor^m}(p,3/j) )&= \Tor^m \setminus B_{\Tor^m}(p,3/j),
  \end{align}
  is the identity map,
  \begin{align}
  F_j(B_{\Tor^m}(p,3/j)\setminus \{p\} )&= B_{\Tor^m}(p,3/j)\setminus B_{\Tor^m}(p,2/j),
  \end{align}
  is defined by radial scaling, and
  \begin{align}
  F_j(p) &= q \in \partial B_{\Tor^m}(p,2/j).
\end{align}
By construction this map is almost onto and now we will argue that this map is almost distance preserving.

\textit{Case 1:} $p_1,p_2 \in \mathbb{D}^m\setminus \partial \mathbb{D}^m \subset M_{\infty}$. 

Connect $F_j(p_1)$ to $F_j(p_2)$ via a straight line $\gamma\subset B_{\Tor^m}(p,1/j)$ and let $\alpha = \alpha_1\alpha_2\alpha_3\alpha_4\alpha_5$ be the concatenation of $\alpha_1,\alpha_5 \subset B_{\Tor^m}(p,1/j)$, $\alpha_3 \subset \partial B_{\Tor^m}(p,2/j)$ length minimizing, and $\alpha_2,\alpha_4 \subset B_{\Tor^m}(p,2/j)\setminus B_{\Tor^m}(p,1/j)$ radial curves. Then we can calculate
\begin{align}
 L_j(\alpha)&\le L_j(\alpha_1)+L_j(\alpha_5)
 \\&\qquad + 2 \int_{1/j}^{2/j} h_j(jr) dr  +\diam(\partial B_{\Tor^m}(p,2/j))\\
 &= L_{\infty}(F_j^{-1}(\alpha_1))+L_{\infty}(F_j^{-1}(\alpha_5))
 \\&\qquad + 2 \int_{1/j}^{2/j} h_j(jr) dr  +\diam(\partial B_{\Tor^m}(p,2/j)).
\end{align}

Hence  we find
\begin{align}
    d_j(F_j(p_1), F_j(p_2))&\le \min\{L_j(\gamma),\inf_{\alpha}L_j(\alpha)\} 
    \\&\le d_{\infty}(p_1,p_2)
    \\&\qquad + 2 \int_{1/j}^{2/j} h_j(jr) dr  +Diam(\partial B_{\Tor^m}(p,2/j)).
\end{align}

Let $\beta_j$ be the length minimizing curve with respect to $g_j$ connecting $F_j(p_1)$ to $F_j(p_2)$. If $\beta_j \subset B_{\Tor^m}(p,1/j)$ then
\begin{align}
    d_{\infty}(p_1,p_2) \le L_{\infty}(F_j^{-1}(\beta_j)) = L_j(\beta_j) = d_j(F_j(p_1),F_j(p_2)).
\end{align}
If not then we can decompose $\beta_j = \beta_j^1\beta_j^2\beta_j^3$ into the concatenation of $\beta_j^1,\beta_j^3 \subset B_{\Tor^m}(p,1/j)$ and $\beta_j^2 \subset \Tor^m\setminus B_{\Tor^m}(p,1/j)$ so that
\begin{align}
    d_{\infty}(p_1,p_2) &\le L_{\infty}(F_j^{-1}(\beta_j^1)) + L_{\infty}(F_j^{-1}(\beta_j^3)) 
    \\&\le L_j(\beta_j) = d_j(F_j(p_1),F_j(p_2)).
\end{align}

\textit{Case 2:} $p_1,p_2 \in \Tor^m \subset M_{\infty}$.

Connect $F_j(p_1)$ to $F_j(p_2)$ via a straight line $\gamma\subset \Tor^m \setminus B_{\Tor^m}(p,2/j)$, if possible, and by $\alpha = \alpha_1\alpha_2\alpha_3$ the concatenation of $\alpha_1,\alpha_3 \subset \Tor^m \setminus B_{\Tor^m}(p,2/j)$ and $\alpha_2 \subset \partial B_{\Tor^m}(p,2/j)$ length minimizing. Then we can calculate
\begin{align}
 L_j(\alpha)&\le L_j(\alpha_1)+L_j(\alpha_3)+ 2 \diam(\partial B_{\Tor^m}(p,2/j))\\
 &\le L_{\mathbb{T}^m}(\alpha_1)+L_{\mathbb{T}^m}(\alpha_3)+ 2 \diam(\partial B_{\Tor^m}(p,2/j)).
\end{align}

Hence  we find
\begin{align}
    d_j(F_j(p_1), F_j(p_2))&\le \min\{L_j(\gamma),\inf_{\alpha}L_j(\alpha)\} 
    \\&\le d_{\infty}(p,q)+\diam(B_{\Tor^m}(p,3/j)) 
    \\&\qquad + 2 \diam(\partial B_{\Tor^m}(p,2/j)).
\end{align}

Let $\beta_j$ be the length minimizing curve with respect to $g_j$ connecting $F_j(p_1)$ to $F_j(p_2)$. Decompose $\beta_j = \beta_j^1\beta_j^2\beta_j^3$ into the concatenation of $\beta_j^1,\beta_j^3 \subset \Tor^m \setminus B_{\Tor^m}(p,2/j)$ and $\beta_j^2 \subset  B_{\Tor^m}(p,2/j)$ ($\beta_j^2$ and $\beta_j^3$ could be trivial if $\beta_j \subset \Tor^m \setminus B_{\Tor^m}(p,2/j)$). Then
\begin{align}
    d_{\infty}(p_1,p_2) &\le L_{\mathbb{T}^m}(\beta_j^1) + L_{\mathbb{T}^m}(\beta_j^3)+\diam(B_{\Tor^m}(p,3/j)) 
    \\&\le L_j(\beta_j) +\diam(B_{\Tor^m}(p,3/j)) 
    \\&= d_j(F_j(p_1),F_j(p_2)) +\diam(B_{\Tor^m}(p,3/j)) .
\end{align}

\textit{Case 3:} $p_1 \in \mathbb{D}^m\setminus \partial \mathbb{D}^m \subset M_{\infty}, p_2 \in \Tor^m \subset M_{\infty}$.

Let $\beta_j$ be the length minimizing curve with respect to $g_j$ connecting $F_j(p_1)$ to $F_j(p_2)$. Then we can decompose $\beta_j = \beta_j^1\beta_j^2\beta_j^3$ into the concatenation of $\beta_j^1 \subset \Tor^m \setminus B_{\Tor^m}(p,2/j)$, $\beta_j^2 \subset B_{\Tor^m}(p,2/j) \setminus B_{\Tor^m}(p,1/j)$, and $\beta_j^3 \subset B_{\Tor^m}(p,1/j)$. Then we can calculate
\begin{align}
    d_{\infty}(p,q) &\le L_{\mathbb{T}^m}(\beta_j^1) + L_{\infty}(F_j^{-1}(\beta_j^3)) +\diam(B_{\Tor^m}(p,3/j)) 
    \\&\le L_j(\beta_j)+\diam(B_{\Tor^m}(p,3/j)) 
    \\&= d_j(F_j(p_1),F_j(p_2))+\diam(B_{\Tor^m}(p,3/j)).
\end{align}

Similarly, if $\beta$ is a curve connecting $F_j(p_1)$ to $F_j(p_2)$, decomposed in a similar way as above, then
\begin{align}
    d_j(F_j(p_1),F_j(p_2)) &\le \inf_{\beta} L_j(\beta)
    \\&=\inf_{\beta} \left\{L_j(\beta_j^1)+L_j(\beta_j^2)+L_j(\beta_j^3) \right\}
    \\&=\inf_{\beta} \left\{L_{\mathbb{T}^m}(\beta_j^1)+L_j(\beta_j^2)+L_{\infty}(F_j^{-1}(\beta_j^3)) \right\}
    \\&\le d_{\infty}(p_1,p_2)+\diam(B_{\Tor^m}(p,3/j))
    \\&\qquad +\diam(\partial B_{\Tor^m}(p,2/j))+\frac{1}{j}\int_1^2 h_j(r) dr.
\end{align}

Since all of the error terms are going to zero uniformly for $p_1, p_2 \in M_{\infty}$ as $j \rightarrow \infty$ we find that $F_j$ is $\delta_j$-almost distance preserving and $\delta_j$-almost onto for $\delta_j \to 0$.  Thus by \cite{Gromov-metric} (cf. Corollary 7.3.28 in \cite{BBI})
 we have that
\begin{align}
    M_j \GHto M_{\infty}.
\end{align}

 Once $M_j$ has a GH limit $M_\infty$ then by Theorem 3.20 of \cite{SW-JDG} we know that a subsequence converges
in the $\mathcal{F}$ sense either to the $0$ space or a subset of the GH limit. Now consider any point $q\in \Tor^m \setminus \bar{B}_{\Tor^m}(p,2/j)$
and $B_{d_j}(q, r) \in M_j$ where $r<d_0(p,q)\le d_j(p,q)$.   For $r$ chosen sufficiently small and $j$ sufficiently large this
is a Euclidean ball, and so by Lemma 4.1 of \cite{Sormani-ArzAsc}, 
choosing possibly a smaller $r>0$, $B_{d_j}(q, r) $ converges to a Euclidean ball of radius $r$ inside
the $\mathcal{F}$ limit.  For $q \in B_{\Tor^m}(p,1/j)$ a similar argument implies that for $r$ small enough and $j$ large enough $B_{d_j}(q,r)$ converges to a Euclidean ball of radius $r$.  Piecing together these Euclidean balls, we see that the intrinsic flat limit is $M_{\infty}$
with possibly one point removed.  But that point is added back in when one takes the metric completion. 
\end{proof}

\subsection{Conformal Tori with no GH nor ${\mathcal{F}}$ limit}

Here we see an example where th $L^p$ norm of the conformal factor $f_j$ diverges for any $p > \frac{m}{\eta}$, $\eta > 1$ and whose volume diverges. We will show that the GH and 
${\mathcal{F}}$ limits are not well defined.

\begin{ex}\label{No Conv}
Consider the sequence of functions on $\Tor^m$ which are radially defined from a point $p \in \Tor^m$
\begin{equation}
f_j(r)=
\begin{cases}
j^{\eta} &\text{ if } r \in [0,1/j]
\\h_j(jr) &\text{ if } r \in [1/j,2/j]
\\ 1 & \text{ if } r \in (1/j,\sqrt{m}\pi].
\end{cases}
\end{equation}
where $\eta > 1$ and $h_j:[1,2] \rightarrow \R$ is a smooth, decreasing function so that $h_j(1) = j^{\eta}$ and $h_j(2) = 1$. Then for $M_j = (\Tor^m, f_j^2 g_{\Tor^m})$
\begin{align}
    &\vol(M_j) \rightarrow \infty,
    \\&\diam(M_j) \rightarrow \infty,
    \\&\|f_j\|_{L^p}\rightarrow \infty, \quad p > \frac{m}{\eta},
\end{align}
 and $M_j$ does not converge in the ${\mathcal{F}}$ or GH sense to any compact metric space.
\end{ex}
\begin{proof}
We begin by computing
\begin{align}
\vol(M_j) &= \int_{\Tor^m}f_j^m dVol 
\\&\ge Vol(B^m(p,1/j))j^{m \eta} + Vol(\Tor^m \setminus B^m(p,2/j))
\\&= j^{m(\eta -1)}Vol(B^m(p,1))+ Vol(\Tor^m \setminus B^m(p,2/j))\rightarrow \infty,
\end{align}
\begin{align}
    \|f_j\|_{L^p}&\ge Vol(B^m(p,1/j))j^{p \eta}  = \omega_m j^{p\eta-m} \rightarrow \infty, \quad p > \frac{m}{\eta},
\end{align}
and
\begin{align}
\diam(M_j) &= \int_0^{\sqrt{m}\pi}f_j dr \label{DiamBoundEx3.3}
\ge j^{\eta-1} + (\sqrt{m} \pi  - 2/j) \rightarrow \infty.
\end{align}
For sake of contradiction assume that $M_j \GHto M_{\infty}$ where $M_{\infty}$ is a compact metric space then the diameter $\diam(M_j)$ must converge to $\diam(M_{\infty})$
but by \eqref{DiamBoundEx3.3} we find that $\diam(M_j) \rightarrow  \infty$ which is a contradiction.

Similarly, for sake of contradiction assume that $M_j \Fto M_{\infty}$. In Lemma 4.1 of \cite{Sormani-ArzAsc} it is shown that if $p_j \in M_j$, $M_j \Fto M_{\infty}$, and $B_j(p_j,R)$, viewed as an integral current space, is such that $B_j(p_j,R) \Fto H(R)$ then $H(R) \subset M_{\infty}$. Notice that $B_j(p,R) = B\left(p, Rj^{-\eta}\right)$ which is contained in the piece of $M_j$ which is Euclidean space and hence $B_j(p,R) \Fto B(p,R) \subset M_{\infty}$.  This implies that $\mass(M_{\infty})\ge \omega_m R^m$ for all $R>0$ which is a contradicts the fact that integral current spaces have finite volume.
\end{proof}

\subsection{Conformal Tori whose GH and $\mathcal{F}$ limit disagree}

Here we see an example where $f_j\rightarrow 1$ in $L^m$ but such that the $L^p$ norm is unbounded for every $ p > m$. This is the borderline case between Example~\ref{L^p Conv} and Example~\ref{No L^m Conv}.  For this reason we will see that the GH limit is not the flat torus.  This shows that we cannot expect the conclusion of Theorem \ref{GenGHandFlatConvMetric} to hold when $f_j$ is unbounded. This also reinforces the fact that we will not be able to find Lipschitz or H\"{o}lder control from above on the distance function when $f_j$ is unbounded. Since a spline develops along the sequence a Lipschitz or H\"older bound from above cannot exist for otherwise we would contradict the main theorem of the appendix in \cite{Allen-Bryden}. 

\begin{ex}\label{VolControlDiamNotConvergent}
Consider the sequence of functions on $\Tor^m$ which are radially defined from a point $p \in \Tor^m$
\begin{equation}
f_j(r)=
\begin{cases}
\frac{j^{\eta}}{1+\ln(j)} &\text{ if } r \in [0,1/j^{\eta}]
\\ \frac{1}{r(1-\ln(r))} &\text{ if } r \in (1/j^{\eta},1/j]
\\ h_j(jr) &\text{ if } r \in (1/j,2/j]
\\ 1 & \text{ if } r \in (2/j,\sqrt{m}\pi].
\end{cases}
\end{equation}
where $\eta > 1$ and $h_j:[1,2]\rightarrow \R$ is a smooth, decreasing function so that $h_j(1) = \frac{j}{1+\ln(j)}$  and $h_j(2) = 1$.
Then $f_j$ is not bounded in $L^p$, $p > m$, but
\begin{align}
    \|f_j-1\|_{L^p(\Tor^m)} \rightarrow 0 \text{ for } p \le m.
\end{align}
For $M_j = (\Tor^m, f_j^2 g_{\Tor^m})$
\begin{align}
M_j &\mGHto M_{\infty},
\end{align}
where $M_{\infty}$ is $\Tor^m$ with a line of length $\ln(\eta)$ attached, and
\begin{align}
M_j &\VFto \Tor^m.
\end{align}
\end{ex}

\begin{proof}
We begin by computing
\begin{align}
&\vol(M_j) = \int_{\Tor^m}f_j^m dV 
\\&= Vol(B^m(p,1/j))\frac{j^m}{(1+\ln(j))^m} 
\\&\qquad+ \int_{B(p,1/j)\setminus B(p,1/j^{\eta})}\lp\frac{1}{r(1-\ln(r))}\rp^m dV
\\&\qquad+ \int_{B(p,2/j)\setminus B(p,1/j)}h_j(jr)^m dV
\\&\qquad+ Vol(\Tor^m \setminus B^m(p,2/j))
\end{align}
\begin{align}
&= \frac{1}{(1+\ln(j))^m}+Vol(\Tor^m \setminus B^m(p,2/j)) 
\\&\qquad+\int_{1/j^{\eta}}^{1/j}\frac{\omega_m}{r(1-\ln(r))^m} dr +\int_{1/j}^{2/j}\omega_m h_j(jr)^mr^{m-1} dr 
\\&= \frac{1}{(1+\ln(j))^m}+Vol(\Tor^m \setminus B^m(p,2/j)) 
\\&\qquad+\int_{1-\ln(1/j^{\eta})}^{1-\ln(1/j)}-\frac{\omega_m}{u^m} du +\int_{1/j}^{2/j}\omega_m h_j(jr)^mr^{m-1} dr 
\\&= \frac{1}{(1+\ln(j))^m}+Vol(\Tor^m \setminus B^m(p,2/j)) 
\\&\qquad+\frac{\omega_m}{m-1} ((1-\ln(1/j^{\eta}))^{1-m}- (1-\ln(1/j))^{1-m} )
\\&\qquad +\int_{1/j}^{2/j}\omega_m h_j(jr)^mr^{m-1} dr \rightarrow Vol(\Tor^m)
\end{align}
where we notice that
\begin{align}
 \frac{\omega_m (2^m-1)}{mj^m} &=\omega_m \int_{1/j}^{2/j}r^{m-1}dr  
 \\&\le \int_{1/j}^{2/j}\omega_m h_j(jr)^mr^{m-1} dr 
 \\&\le \omega_m \frac{j^m}{(1+\ln(j))^m}\int_{1/j}^{2/j}r^{m-1}dr =  \frac{\omega_m (2^m-1)}{m(1+\ln(j))^m}.
\end{align}
Now we also calculate the diameter
\begin{align}
\diam(M_j) &= \int_0^{\sqrt{m}\pi}f_j dr \label{DiamBoundEx3.4}
\\&= \frac{1}{1+\ln(j)}+\int_{1/j^{\eta}}^{1/j} \frac{1}{r(1-\ln(r))} dr 
\\&\qquad+\int_{1/j}^{2/j} h_j(jr) dr + (\sqrt{m}\pi - 2/j) 
\\&=\frac{1}{1+\ln(j)}+ (\sqrt{m}\pi - 2/j) +\int_{1/j}^{2/j} h_j(jr) dr
\\&\qquad+\ln(1-\ln(1/j^{\eta})) - \ln(1-\ln(1/j)) 
\\&=\frac{1}{1+\ln(j)}+ (\sqrt{m}\pi - 2/j) +\int_{1/j}^{2/j} h_j(jr) dr
\\&\qquad+\ln\lp\frac{1-\ln(1/j^{\eta})}{1-\ln(1/j)} \rp \rightarrow \sqrt{m} \pi + \ln(\eta).
\end{align}
So in this case we see that the volume and diameter are controlled but the diameter does not converge to the diameter of $\Tor^m$.

We also note that this example does not converge in $L^p$ norm for any $p > m$ since
\begin{align}
\int_{\Tor^m}f_j^p dVol &\ge Vol(B^m(p,\frac{1}{j^{\eta}}))\frac{j^{\eta p}}{(1+\ln(j))^p} 
=\frac{j^{\eta(p-m)}}{(1+\ln(j))^p} \rightarrow \infty.
\end{align}
Let $M_{\infty} = \Tor^m \disjointunion [0,\ln(\eta)]$ attached at a point and consider the map 
\begin{align}
    F_j: M_{\infty} \rightarrow M_j
\end{align}
defined so that $F_j(\Tor^m) = \Tor^m \setminus B(p,1/j^{\eta})$ in a standard almost distance preserving fashion. Then $F_j([0,\ln(\eta)])$ is mapped onto a radial curve from p to $\partial B(p,1/j^{\eta})\subset M_j$ so that
\begin{align}
    F_j(r) = q, \quad q \in \partial B(p,1/j^{\eta/e^r}).
\end{align} 
Notice that this map is almost onto. We can see that the map is almost distance preserving since for $q \in \partial B(p,1/j^{\eta/e^r})\subset M_j$ where $r \le \ln(\eta)$ we can perform a calculation similar to \eqref{DiamBoundEx3.4}
\begin{align}
d_j(p,q) &= \int_0^{1/j^{\eta/e^r}}f_j dr 
\\&=\frac{1}{1+\ln(j)}+\ln\lp\frac{1-\ln(1/j^{\eta})}{1-\ln(1/j^{\eta/e^r})} \rp \rightarrow r.
\end{align} 
Hence $M_j \GHto M_{\infty}$ by \cite{Gromov-metric}.

Once $M_j$ has a GH limit $M_\infty$ then by \cite{SW-JDG} we know that a subsequence converges
in the $\mathcal{F}$ sense either to the $0$ space or a subset of the GH limit.   Since the $\mathcal{F}$
limit must have the same dimension as the sequence, the subsequence converges
in the $\mathcal{F}$ sense either to the $0$ space or a subset of the flat torus.   Now consider any point $q\neq p$ in
the flat torus and consider $B_{d_j}(q, r) \in M_j$ where $r<d_0(p,q)\le d_j(p,q)$.   For $j$ sufficiently large this
is a Euclidean ball, and so by Lemma 4.1 in \cite{Sormani-ArzAsc}, choosing possibly a smaller $r>0$
$B_{d_j}(q, r) $ converges to a Euclidean ball of radius $r$ inside
the $\mathcal{F}$ limit.  Piecing together these Euclidean balls, we see that the intrinsic flat limit is a flat torus
with possibly one point removed.  But that point is added back in when one takes the metric completion and hence we find $\mathcal{VF}$ convergence.   A new proof of this limit will appear in upcoming work of the authors.
\end{proof}

\subsection{No GH Convergence}

Here we build on the previous example by producing an example which does not even have a GH limit. This shows why we cannot expect to find GH convergence when we only assume a metric lower bound and volume convergence. Again, $f_j\rightarrow 1$ in $L^m$ but such that the $L^p$ norm is unbounded for every $ p > m$ which is the borderline case between Example~\ref{L^p Conv} and Example~\ref{No L^m Conv}. This example appears to converge in the $\mathcal{VF}$ sense to a flat torus and this will be shown in a future paper by the authors.

\begin{ex}\label{FConvNoGHConv}
Consider the sequence of functions on $\Tor^m$, $m \ge 2$, which are radially defined from a point $p \in \Tor^m$
\begin{equation}
h_j^p(r)=
\begin{cases}
\frac{j^{\eta}}{1+\ln(j)} &\text{ if } r \in [0,1/j^{\eta}]
\\ \frac{1}{r(1-\ln(r))} &\text{ if } r \in (1/j^{\eta},1/j]
\\ h_j(jr) &\text{ if } r \in (1/j,2/j],
\end{cases}
\end{equation}
where $\eta > 1$ and $h_j:[1,2]\rightarrow \R$ is a smooth, decreasing function so that $h_j(1) = \frac{j}{1+\ln(j)}$  and $h_j(2) = 1$.
Let $\{p_1,p_2,...,p_{k_j}\}$, where $k_j \in \N$, $k_j \le \lfloor \sqrt{\ln(j)} \rfloor$, $k_j \rightarrow \infty$, by equally spaced points around a flat circle in $\Tor^m$ and define the conformal factor 
\begin{equation}
f_j(q)=
\begin{cases}
h_j^{p_i}(q) &\text{ if } q \in B(p_i,2/j), 1 \le i \le k_j
\\ 1 &\text{ otherwise}.
\end{cases}
\end{equation}
Then $f_j$ is not bounded in $L^p$, $p > m$,
\begin{align}
    \|f_j-1\|_{L^p(\Tor^m)} \rightarrow 0 \text{ for } p \le m,
\end{align}

but the sequence $M_j = (\Tor^m, f_j^2g_{\Tor^m})$ has no GH limit.
\end{ex}
\begin{proof}
First we verify convergence in $L^m$ norm
\begin{align}
\vol(M_j) &= \int_{\Tor^m}f_j^m dVol
\\& = \sum_{i=1}^{k_j} \int_{B(p_i,2/j)} (h_j^{p_i})^m dVol 
\\&\qquad + \vol(\Tor^m)- k_j \vol(B(p_1,2/j))
\\& = \vol(\Tor^m)- k_j \vol(B(p_1,2/j))+\frac{k_j}{(1+\ln(j))^m} 
\\&\qquad +\frac{\omega_m k_j}{m-1} ((1+\eta\ln(j))^{1-m}- (1+\ln(j))^{1-m} )
\\&\qquad +k_j\int_{1/j}^{2/j}\omega_m h_j(jr)^mr^{m-1} dr \rightarrow \vol(\Tor^m)
\end{align}
where we took advantage of calculation done in the proof of Examples \ref{VolControlDiamNotConvergent} and the fact that $k_j \le \lfloor \sqrt{\ln(j)}\rfloor$. 
Now we find for $q \in \partial B(p_k,1/j^{\eta/e^r})$, $r\le \ln(\eta)$ we can perform a calculation similar to \eqref{DiamBoundEx3.4} to find
\begin{align}
d_j(p_k,q) &= \int_0^{1/j^{\eta/e^r}}f_j dr \rightarrow r.
\end{align}
 So for $r= \ln(\eta)$ we find that $N(r) \ge k_j$, the number of balls of radius $r$ needed to cover $M_j$. Hence by the fact that $k_j \rightarrow \infty$ and Gromov's compactness theorem \cite{Gromov-metric} we find that the sequence cannot have a GH limit.
\end{proof}

\section{Proof of the Main Theorem}\label{sec:MainThmProof}

In this section we will use the hypotheses of Theorem \ref{GenGHandFlatConvMetric}  to achieve important estimates on distances, volume and diameter which will allow us to conclude with the proof of Theorem \ref{ConfGHandFlatConv} and Theorem \ref{GenGHandFlatConvMetric}.   Note that
in each step of the proof we carefully itemize exactly which hypotheses are needed for that particular step, in this way our lemmas and propositions may be applied elsewhere.

\subsection{Global Volume Convergence to Local Volume Convergence}

In this subsection we prove:

\begin{lem}\label{Global-to-Local}  
Let $g_j, g_0$ be continuous Riemannian metrics defined on the compact manifold $M$. If 
\begin{align}
    g_j \ge \lp1 - 1/j\rp g_0
\end{align}
and 
\begin{align}
   \vol_{g_j}(M)\to \vol_{g_0}(M)
\end{align}  
then for any measurable set $W \subset M$ we find
\begin{align}
\vol_{g_j}(W)&\rightarrow \vol_{g_0}(W).
\end{align}
\end{lem}

We start by showing that the volume of measurable sets $U \subset M$ with respect to $g_j$ converge to the volume with respect to $M_0$. Our first lemma considers the lower bound on volume for the sequence. 

\begin{lem}\label{ConfVolBoundBelow} 
Let $g_j, g_0$ be Riemannian metrics defined on the compact manifold  $M$. If 
\begin{align}
   g_j \ge \lp 1 - 1/j\rp g_0 
\end{align}
then for any measurable set $U \subset M$ we find
\begin{align}
\vol((U,g_j)) &\ge (1 - 1/j)^{\frac{m}{2}} \vol((U,g_0)).
\end{align}
\end{lem}

\begin{proof}
We note that by the assumption $g_j \ge g_0 \lp1 - 1/j\rp$ we immediately find the inequality for the associated measures $dvol_{g_j} \ge (1 - 1/j)^{m/2} dvol_{g_0}$ from which the claim follows.
\end{proof}

We now prove Lemma~\ref{Global-to-Local}

\begin{proof}
Given any measurable set $W$ and our lower bound on $g_j$ we
can apply Lemma~\ref{ConfVolBoundBelow} to both $W$ and $M\setminus W$
to obtain 
\begin{align}
\vol((W,g_j)) &\ge (1 - 1/j)^{\frac{m}{2}} \vol((W,g_0))\\
\vol((M\setminus W,g_j)) &\ge (1 - 1/j)^{\frac{m}{2}} \vol((M\setminus W,g_0)).
\end{align}
So we also have an upper bound on 
\begin{eqnarray}
\vol((W,g_j)) &= & \vol_{g_j}(M) - \vol((M\setminus W,g_j)) \\
&\le&  \vol_{g_j}(M) - (1 - 1/j)^{\frac{m}{2}} \vol((M\setminus W,g_0))
\end{eqnarray}
Applying the volume convergence we obtain our claim.
\end{proof}

\subsection{Convergence in $L^{\frac{m}{2}}$ norm to Volume Convergence}\label{subsec:VolControl}

In this subsection we prove that convergence in $L^{\frac{m}{2}}$ norm  for $m=dim(M)$ combined with $g_j \ge (1-1/j)g_0$
 implies volume convergence.   Note that this theorem does not require $g_j\le Kg_0$.

\begin{lem}\label{ConfVolBound}
Let $g_j, g_0$ be continuous Riemannian metrics defined on the compact manifold $M$. If 
\begin{align}
    g_j \ge \lp1 - 1/j\rp g_0
\end{align}
and 
\begin{align}
    \|g_j\|_{L^{\frac{m}{2}}_{g_0}(M)} \rightarrow \|g_0\|_{L^{\frac{m}{2}}_{g_0}(M)},
\end{align}  then for any measurable set $U \subset M$ we find
\begin{align}
\vol_j(U)&\rightarrow \vol_0(U).
\end{align}
\end{lem}
\begin{proof}
First notice
\begin{align}
  \|g_j\|_{L^{\frac{m}{2}}_{g_0}(U)}^{\frac{m}{2}}&= \|g_j\|_{L^{\frac{m}{2}}_{g_0}(M)}^{\frac{m}{2}} - \|g_j\|_{L^{\frac{m}{2}}_{g_0}(M\setminus U)}^{\frac{m}{2}}
    \\&\le \|g_j\|_{L^{\frac{m}{2}}_{g_0}(M)}^{\frac{m}{2}} - \|g_0\|_{L^{\frac{m}{2}}_{g_0}(M\setminus U)}^{\frac{m}{2}} 
    \\&\rightarrow \|g_0\|_{L^{\frac{m}{2}}_{g_0}(U)}^{\frac{m}{2}} = m^{\frac{m}{4}}\vol_0(U).
\end{align}
Now by the determinant trace inequality which follows from the arithmetic-geometric mean inequality we find
\begin{align}
Det_{g_0}(g_j)^{\frac{1}{m}} \le \frac{1}{m} Tr_{g_0}(g_j)\le\frac{|g_j|_{g_0}}{m^{\frac{1}{2}}}.
\end{align}
Here $Det_{g_0}(g_j), Tr_{g_0}(g_j)$ means to choose an orthonormal basis $\{v_1,...,v_m\}$ for $g_0$ and form a symmetric, positive definite matrix by evaluating $A_{pq} = g_j(v_p,v_q)$, $1 \le p,q\le m$ and then take the determinant or trace of that matrix, respectively. Hence if $\lambda_1^2,...,\lambda_m^2$ are the eigenvalues of $g_j$ with respect to $g_0$ then,
\begin{align}
 Tr_{g_0}(g_j)=\sum_{i=1}^m \lambda_i^2 \le \sqrt{m}|g_j|_{g_0},
\end{align}
 and hence we obtain
\begin{align}
Det_{g_0}(g_j) \le \frac{|g_j|^{m}_{g_0}}{m^{\frac{m}{2}}}.
\end{align}
Using this inequality we find
\begin{align}
\vol_j(U) &= \int_{U} \sqrt{Det_{g_0}(g_j)} d vol_{g_0}
\\&\le  \int_{U}\frac{|g_j|^{\frac{m}{2}}_{g_0}}{m^{\frac{m}{4}}} d vol_{g_0}\rightarrow \vol_0(U).
\end{align}

Now we have an upper volume bound and the fact that $g_j \ge \lp1-1/j \rp g_0$ means
we can apply Lemma \ref{ConfVolBoundBelow} for the lower bound.
\end{proof}

\subsection{Pointwise Almost Everywhere Convergence of Distances}

This next theorem uses the volume convergence from above and distance controls from below to
imply distances converge almost everywhere on $M\times M$.  It does not require the
Lipschitz bound from above and will be applied in future work of the authors.

\begin{thm}\label{PointwiseAEConvergenceDistances}
If $(M, g_j)$ are compact continuous Riemannian manifolds without boundary and $(M,g_0)$ is a smooth Riemannian manifold such that
\be\label{MetricLowerBoundPointwiseThm}
g_j(v,v) \ge g_0(v,v) \qquad \forall v\in TM
\ee
and
\be
\vol_j(M) \to \vol_0(M),
\ee
then there exists a subsequence so that
\be
\lim_{j\to \infty} d_j(p,q) = d_0(p,q) \textrm{ pointwise a.e. } (p,q) \in M\times M, \label{pointwiseAlmostEverywhereDistances}
\ee
where
\be
d_j(p,q)=\inf\{ L_j(C):\, \,C(0)=p,\,\,C(1)=q\,\},
\ee
\be
L_j(C)=\int_0^1 \sqrt{g_j\lp C'(s), C'(s)\rp} \, ds.
\ee
\end{thm}

\begin{proof}
First we note by Lemma \ref{biLip-bounds} that
\be
d_0(p,q) \le d_j(p,q) \qquad \forall (p,q) \in M\times M.
\ee

Applying our volume convergence and Lemma~\ref{Global-to-Local} to any measurable set $U \subset M$ we have:
\be\label{VolumeSetsConverge}
    \vol_j(U) = \int_U Det_{g_0}(g_j) dV_{g_0} \rightarrow \vol_0(U),
\ee
which will be key to the argument which follows.

Let $p,q \in M$ so that $q$ is not a cut point of p with respect to $g_0$.  Let $v_{pq}\in T_pM$ be the smallest vector such that $\exp_p(v_{pq}) = q$.   Note that if $q$ is a cut point then we can choose to replace it with another point which is arbitrarily close to $q$ and closer to $p$. This does not cause a problem since the set of cut points from $p$ has measure zero. Our goal is to show pointwise a.e. convergence of distances for a subsequence.   
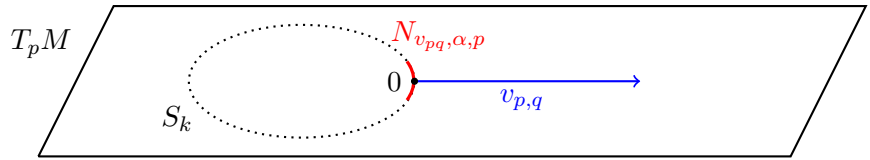
\begin{figure}[h] 
\centering
\begin{tikzpicture}[scale=.5]
\draw[thick] (-18,-5) to (2, -5) to (4,-1) to (-16, -1) to (-18, -5);
\node[right] at (-19, -2) {$T_pM$};
\draw [->, blue, thick] (-8,-3) -- (-2,-3);
\node[right] at (-6, -3.5) {$\textcolor{blue}{v_{p,q}}$};
\draw[dotted, thick] (-11,-3) circle [x radius=3, y radius=1.5];
\node[right] at (-15, -4) {$S_k$};
\draw[very thick,red] (-8.2,-3.5) arc (-20:20:3 and 1.5);
\node[right] at (-8.9,-1.8) {$\textcolor{red}{N_{v_{pq},\alpha,p}}$};
\draw[fill] (-8,-3) circle [radius=0.08];
\node[right] at (-9, -3) {$0$};
\end{tikzpicture}
\label{fig-A}
\caption{ $N_{v_{pq},\alpha,p}=\{w \in S_k: |w|_{g_0} < \alpha\}\subset T_pM$.}
\end{figure}

Given any $v\in T_pM$ we can define $S_k\subset T_pM$ be a $m-1$ sphere (or hyperplane if $k=0$) with constant principal curvature $k$ as a subset of $T_pM$, so that $0\in S_k$ and $v \perp T_0S_k$. 
See Figure~\ref{fig-A} where $v=v_{pq}$.  Here we define the principal curvatures of $S_k$ relative to the normal vector $v$. For instance, $S_0 = v^{\perp}\subset T_pM$. The parameter $k$ is important to avoid focal points later in the argument.  Now for $\alpha \in (0,\infty)$ let
\begin{align}
    N_{v,\alpha,p}=\{w \in S_k: |w|_{g_0} < \alpha\}.
\end{align}

\begin{figure}[h] 
\centering
\begin{tikzpicture}[scale=.5]
\draw[thick] (-20,-5) to [out=5,in=195] (0, -5) 
to [out=45,in=210] (4,-1) 
to [out=170,in=10](-16, -1) 
to [out=230,in=70](-20, -5);
\node[right] at (-20, -2) {$M$};
\draw[very thick] (-2.2,-3.5) arc (-20:20:3 and 1.5);
\draw [thick] (-8.2,-3.5) to [out=340,in=190] (-2.2,-3.5);
\draw (-8.15,-3.4) to [out=345,in=185] (-2.15,-3.35);
\draw (-8.1,-3.3) to [out=350,in=180] (-2.1,-3.2);%
\draw (-8.05,-3.15) to [out=355,in=175] (-2.05,-3.1);
\draw [thick, blue] (-8,-3) to [out=0,in=170] (-2,-3);
\node[right] at (-5, -4.5) {$\textcolor{blue}{\gamma_{v_{p,q}}}$};
\draw  (-8.05,-2.85) to [out=8,in=170] (-2.05,-2.85);
\draw  (-8.1,-2.7) to [out=15,in=170] (-2.1,-2.7);%
\draw  (-8.15,-2.6) to [out=20,in=165] (-2.15,-2.6);
\draw [thick] (-8.2,-2.5) to [out=25,in=160] (-2.2,-2.5);
\node[right] at (-4,-1.4) {$\textcolor{black}{\mathcal{T}_{v_{p,q}, \alpha,p} }$};
\draw[very thick,red] (-8.2,-3.5) arc (-20:20:3 and 1.5);
\node[right] at (-12,-1.5) {$\textcolor{red}{ \exp_p(N_{v_{pq}, \alpha,p})}$};
\draw[fill] (-8,-3) circle [radius=0.08];
\node[right] at (-9, -3) {$p$};
\draw[fill] (-2,-3) circle [radius=0.08];
\node[right] at (-2, -3) {$q$};
\draw [thick, orange] (-8.1,-3.3) to [out=350,in=180] (-2.1,-3.2);%
\draw[fill, orange] (-8.1,-3.3) circle [radius=0.06];
\node[right] at (-9.1, -3.7) {$\textcolor{orange}{p'}$};
\draw[fill, orange]  (-2.05,-3.23) circle [radius=0.06];
\node[right] at (-2.4, -3.8) {$\textcolor{orange}{q'}$};
\node[right] at (-7, -4.5) {$\textcolor{orange}{\gamma_{p'}}$};
\end{tikzpicture}
\label{fig-B}
\caption{$\mathcal{T}_{v_{pq}, \alpha,p}  \subset M$.}
\end{figure}
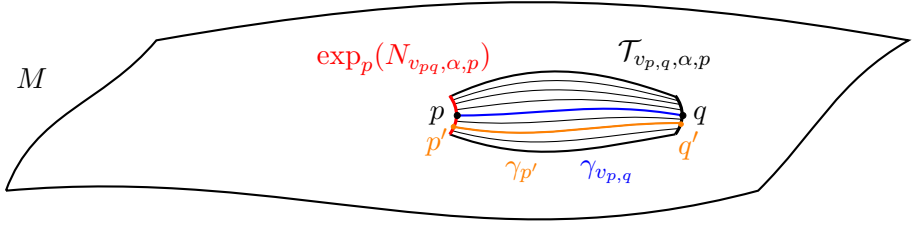

Now if $p'= \exp_p(w)$, $w \in N_{v,\alpha,p}$ and we choose $\alpha$ small enough then we can extend $v$ to  $T_{p'}M$ for all $p' \in \exp_p(N_{v, \alpha,p})$, by choosing a vector $v \in T_{p'}M$ so that  $v \perp \exp_p(S_k)$, of the same length as $v \in T_pM$ and so that $v$ is a continuous vector field on $\exp_p(N_{v, \alpha,p})$. This allows us to define $q'=\exp_{p'}(v)$ as well as
\begin{align}
    \gamma_{p'}(t) = \exp_{p'}(tv), \quad 0 \le t \le 1,
\end{align}
and
\begin{align}
    \mathcal{T}_{v, \alpha,p} = \{\gamma_{p'}(t):p' = \exp_p(w), w \in N_{v,\alpha,p}, 0 \le t \le 1\}.
\end{align}
See Figure~\ref{fig-B}.
Now we can choose $\alpha$ small enough so that $\mathcal{T}_{v,\alpha,p}$ is foliated by $\gamma_{p'}$ which are length minimizing with respect to $g_0$. If we let 
\begin{align}
 \exp^{\perp}: N\exp_p(S_k) \rightarrow M   
\end{align} 
be the normal exponential map where $ N\exp_p(S_k)$ is the normal bundle to $ \exp_p(S_k)\subset M$, $d \mu_{N_{v,\alpha,p}}=d\mu_N$ be the usual measure for $N_{v_{pq},\alpha,p} \subset T_pM \approx\R^m$, and $\lambda_1^2,...,\lambda_m^2$ the eigenvalues of $g_j$ with respect to $g_0$ where $\lambda_1 \le ... \le \lambda_m$. Let $\sqrt{h}$ be the square root of the determinant of the metric $h$ for the hypersurface $\exp(N_{v,\alpha,p})$ in normal coordinates on $N_{v,\alpha,p}$. Let $d\exp^{\perp}$ be the differential of the normal exponential map and $|d\exp^{\perp}|_{g_0}$ be the determinant of this map evaluated in directions orthogonal to the $\gamma_{p'}$ which foliate $\mathcal{T}_{v_{pq},\alpha,p}$. Then by the coarea formula we can calculate 
\begin{align}
  \vol_j(\mathcal{T}_{v,\alpha,p})&= \int_{\mathcal{T}_{v, \alpha}}  \sqrt{Det_{g_0}(g_j)}dV_{g_0} 
   \\&= \int_{N_{v,\alpha,p}}\int_{\gamma_{p'}}\lambda_1...\lambda_m |d\exp^{\perp}|_{g_0}\sqrt{h}dt_{g_0}d\mu_{N}
   \\& \ge  \int_{N_{v,\alpha,p}}\int_{\gamma_{p'}}\lambda_1...\lambda_{m-1} |d\exp^{\perp}|_{g_0}\sqrt{h}dt_{g_j}d\mu_{N}
    \\&\ge  \int_{N_{v,\alpha,p}}\int_{\gamma_{p'}} |d\exp^{\perp}|_{g_0}\sqrt{h}dt_{g_j}d\mu_{N}\label{CrucialVolumeToDistanceEq1} 
   \\&\ge  \int_{N_{v,\alpha,p}}\int_{\gamma_{p'}} |d\exp^{\perp}|_{g_0}\sqrt{h}dt_{g_0}d\mu_{N}\label{CrucialVolumeToDistanceEq2} 
   \\&=  \vol_0(\mathcal{T}_{v,\alpha,p}),
\end{align}
where we used the assumption that $g_0 \le g_j$ to deduce that 
\begin{align}
    dt_{g_0} &= |\gamma'|_{g_0} dt \le |\gamma'|_{g_j} dt =  dt_{g_j},
    \\ dt_{g_j} &= |\gamma'|_{g_j} dt \le \lambda_m |\gamma'|_{g_0} dt = \lambda_m dt_{g_0}
    \\ \lambda_i &\ge 1 \quad 1 \le i \le m.
\end{align}

Then since we have assumed that $Vol_j(M) \rightarrow Vol_0(M)$ we see by \eqref{VolumeSetsConverge} that
\begin{align}\label{FoliationVolume}
    \vol_j(\mathcal{T}_{v,\alpha,p})\rightarrow \vol_0(\mathcal{T}_{v,\alpha,p}),
\end{align}
which squeezes \eqref{CrucialVolumeToDistanceEq1} and \eqref{CrucialVolumeToDistanceEq2} to converge to each other. 

In normal coordinates we notice that $\sqrt{h}=1$ at $p$ and hence we can choose $\alpha$ small enough if needed to ensure that $\sqrt{h} > h_0 > 0$ on $N_{v,\alpha,p}$.

Now we notice that if there are no focal points between $S_0=v^{\perp}$ and $\gamma_p(tv)$ for all $0 \le t \le |v|_{g_0}$ then by construction of $\mathcal{T}_{v,\alpha,p}$ and Proposition 10.30 of \cite{Oneill} we have 
\begin{align} 
    |d\exp^{\perp}|_{g_0} \ge A_{p,q} \quad \text{on } \mathcal{T}_{v,\alpha,p}, \label{NonDegnerateFoliation}
\end{align}
where $A_{p,q}$ is a constant which depends on the spreading of geodesics foliating $\mathcal{T}_{v,\alpha,p}$. 

If it happens to be the case that there are focal points  between $S_0=v^{\perp}$ and $\gamma_p(tv)$ for some $0 \le t \le |v|_{g_0}$ then we can choose $k <0$ to be small enough so that the index form must be positive definite along $\gamma_p(tv)$ (Note that $\exp_p(S_k)$ will not have constant principal curvatures but the principal curvatures will agree with $S_k$ at $p$ and the difference in curvatures is determined by the metric and its Christoffel symbols in normal coordinates at $p$). This can be done since we know the curvature of $g_0$ is bounded (See the expression for the index form in \cite{Oneill} Corollary 10.27). Hence if we choose $k$ small enough we ensure that the geodesics which leave $\exp_p(S_k)$ diverge from each other enough so that they cannot meet up at a focal point before $q$ or even $C(p)$, the cut locus of $p$. So by Proposition 10.30 of \cite{Oneill} we know that \eqref{NonDegnerateFoliation} holds.

Now by combining this observation with \eqref{FoliationVolume} we find
\begin{align}
    \int_{N_{v,\alpha,p}}&\int_{\gamma_{p'}} |d\exp^{\perp}|_{g_0}\sqrt{h}(dt_{g_j}-dt_{g_0})\,d\mu_{N} 
    \\&\ge A_{p,q} h_0  \int_{N_{v,\alpha,p}} L_j(\gamma_{p'}) - L_0(\gamma_{p'})\,d \mu_N
    \\&\ge A_{p,q} h_0  \int_{N_{v,\alpha,p}} d_j(p',q') - d_0(p',q') \,d \mu_N
    \\&= A_{p,q}  h_0 \int_{N_{v,\alpha,p}} |d_j(p',q') - d_0(p',q')| \,d \mu_N, \label{FirstIntegralDistancesToZero}
\end{align}
and hence \eqref{FirstIntegralDistancesToZero} converges to $0$ as $j \rightarrow \infty$.

Notice that we can choose $\bar{\eta}_2,\bar{\eta}_1>0$ so that $\bar{\eta_1} < 1< \bar{\eta}_2$ small enough so that for all $0 <\bar{\eta}_1<\eta < \bar{\eta}_2$ the above argument works for any $\mathcal{T}_{\eta v, \alpha,p}$
 and additionally we note the bound
\begin{align}
    \int_{N_{v,\alpha,p}}&\int_{\gamma_{p'}} |d\exp_{p'}^{\perp}|_{g_0}\sqrt{h}(dt_{g_j}-dt_{g_0})\,d\mu_{N}
    \\&\le Vol_j(\mathcal{T}_{\eta v,\alpha,p})-Vol_0(\mathcal{T}_{\eta v,\alpha,p}) 
    \\&\le Vol_j(M) \le V_0.
\end{align}
Now we can use the dominated convergence theorem to see that
\begin{align}
    \int_{\bar{\eta}_1}^{\bar{\eta}_2}\int_{N_{ v,\alpha,p}}\int_{\gamma_{p'}} |d\exp_{p'}^{\perp}|_{g_0}\sqrt{h}(dt_{g_j}-dt_{g_0})\,d\mu_{N} \,d\eta \rightarrow 0.
\end{align}
This implies for $q'_{\eta} = \exp_{p'}(\eta v)$ that
\begin{align}
    \int_{\bar{\eta}_1}^{\bar{\eta}_2}\int_{N_{ v,\alpha,p}} |d_j(p',q'_{\eta}) - d_0(p',q'_{\eta})| \,d \mu_N \,d \eta \rightarrow 0.
\end{align}
For sufficiently small $\bar{\tau}>0$ and any $\tau \in (-\bar{\tau}, \bar{\tau})$, we define 
\be\label{p-tau}
p_{\tau} = \exp_p(\tau v) \textrm{ and } p_{\tau}' \in N_{v, \alpha, p_{\tau}}
\ee
and then have 
\be \label{q-eta}
q_{\eta}=q_{\eta,\tau} = \exp_{p_\tau}(\eta v) \textrm{ and } q'_\eta=q'_{\eta,\tau}= \exp_{p_\tau'}(\eta v)
\ee
as in Figure~\ref{fig-C}.  

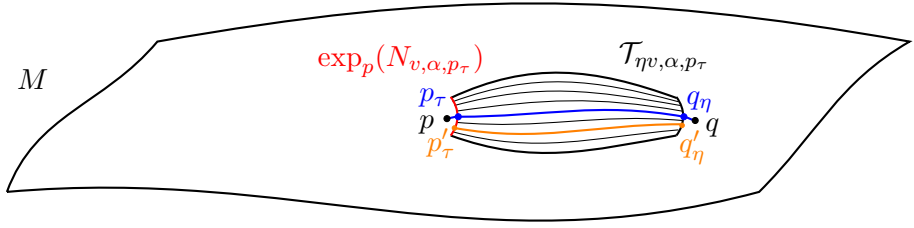
\begin{figure}[h] 
\centering
\begin{tikzpicture}[scale=.5]
\draw[thick] (-20,-5) to [out=5,in=195] (0, -5) 
to [out=45,in=210] (4,-1) 
to [out=170,in=10](-16, -1) 
to [out=230,in=70](-20, -5);
\node[right] at (-20, -2) {$M$};
\draw[thick] (-2.2,-3.5) arc (-20:20:3 and 1.5);
\draw [thick] (-8.2,-3.5) to [out=340,in=190] (-2.2,-3.5);
\draw (-8.15,-3.4) to [out=345,in=185] (-2.15,-3.35);
\draw [thick, orange] (-8.1,-3.3) to [out=350,in=180] (-2.1,-3.2);%
\draw (-8.05,-3.15) to [out=355,in=175] (-2.05,-3.1);
\draw [thick, blue] (-8,-3) to [out=0,in=170] (-2,-3);
\draw  (-8.05,-2.85) to [out=8,in=170] (-2.05,-2.85);
\draw  (-8.1,-2.7) to [out=15,in=170] (-2.1,-2.7);%
\draw  (-8.15,-2.6) to [out=20,in=165] (-2.15,-2.6);
\draw [thick] (-8.2,-2.5) to [out=25,in=160] (-2.2,-2.5);
\node[right] at (-4,-1.4) {$\textcolor{black}{\mathcal{T}_{\eta v, \alpha,p_\tau} }$};
\draw[thick,red] (-8.2,-3.5) arc (-20:20:3 and 1.5);
\node[right] at (-12,-1.5) {$\textcolor{red}{ \exp_p(N_{v, \alpha,p_\tau})}$};
\draw[fill, blue] (-8,-3) circle [radius=0.08];
\node[right] at (-9.3, -2.5) {$\textcolor{blue}{p_\tau}$};
\draw [thick, blue] (-8.3,-3.05) to [out=10, in=180] (-8,-3);
\draw[fill] (-8.3,-3.05) circle [radius=0.08];
\node[right] at (-9.3, -3.2) {$p$};
\draw[fill, orange] (-8.1,-3.3) circle [radius=0.06];
\node[right] at (-9.1, -3.7) {$\textcolor{orange}{p'_\tau}$};
\draw[fill, blue] (-2,-3) circle [radius=0.08];
\node[right] at (-2.2, -2.6) {$\textcolor{blue}{q_\eta}$};
\draw [thick, blue] (-2,-3) to [out=350, in=170] (-1.7, -3.1);
\draw[fill] (-1.7,-3.1) circle [radius=0.08];
\node[right] at (-1.7, -3.3) {$q$};
\draw[fill, orange]  (-2.05,-3.23) circle [radius=0.06];
\node[right] at (-2.4, -3.8) {$\textcolor{orange}{q'_\eta}$};

\end{tikzpicture}
\label{fig-C}
\caption{$\mathcal{T}_{\eta v, \alpha,p_\tau}  \subset M$ when $\tau>0$, $\eta<1$, and $v=v_{pq}$.}
\end{figure}

We can choose $\bar{\tau}$ small enough so that the entire argument above goes through to find
\begin{align}
    \int_{-\bar{\tau}}^{\bar{\tau}}\int_{\bar{\eta}_1}^{\bar{\eta}_2}\int_{N_{ v,\alpha,p_{\tau}}} |d_j(p_{\tau}',q'_{\eta}) - d_0(p_{\tau}',q'_{\eta})| \,d \mu_N \,d \eta \,d \tau\rightarrow 0.
\end{align}
We define 
\begin{align}
    V_{\epsilon}=\{v'\in T_pM : |v'|_{g_0}=|v_{pq}|_{g_0}, g_0(v',v) > (1-\epsilon) |v|_{g_0}^2\}.
\end{align}
Then again for $\epsilon$ chosen small enough we find
\begin{align} \label{almost-there}
    \int_{V_{\epsilon}}\int_{-\bar{\tau}}^{\bar{\tau}}\int_{\bar{\eta}_1}^{\bar{\eta}_2}\int_{N_{v',\alpha,p_{\tau}}} |d_j(p_{\tau}',q'_{\eta}) - d_0(p_{\tau}',q'_{\eta})| \,d \mu_N \,d \eta \,d \tau \,dV\rightarrow 0.
\end{align}

Let us define 
\begin{align}
    &\mathcal{N}(p,q) =
    \\&\quad \{(v',\tau,\eta,w): v' \in V_{\epsilon}, \tau \in (-\bar{\tau},\bar{\tau}), \eta \in (\bar{\eta}_1,\bar{\eta}_2), w \in N_{v',\alpha,p_\tau}\}.
\end{align}
We claim that when $q$ is not in the
cut locus of $p$ and $\epsilon, \bar{\tau}, \bar{\eta}_i$ are well chosen depending on $p$ and $q$
as above depending on $k,\alpha$ which were chosen so that we have (\ref{NonDegnerateFoliation}), then
the map
\be
\Phi_{p,q}: \mathcal{N}(p,q) \to M\times M \qquad \Phi(v',\tau,\eta,w)=(p'_\tau, q'_\eta)
\ee
as in (\ref{p-tau})-(\ref{q-eta}) is a diffeomorphism onto its image 
\be
\mathcal{U}(p,q)=\Phi(\mathcal{N}(p,q))
\ee
which contains the point $(p,q)$.   This can be proven by applying the inverse function theorem
to $\Phi_{p,q}$ at the point $(v_{p,q}, 0, 1, 0)$ as follows.    Take any 
collection of curves $v_i'(t)\in V_{\epsilon}$ for $i=1 ... m-1$
such that $v'(0)=v_{pq}$ and $dv_i'/dt (0)$ are orthonormal and any collection
of curves $w_i(t)\in N_{v_{pq},\alpha,p}$ for $i=1 ... m-1$
such that $w_i(0)=0$ and $dw_i/dt (0)$ are orthnonormal.
Then we have $2m=(m-1)+1+1+(m-1)$ linearly independent vectors:
\begin{eqnarray}
\tfrac{d}{dt}\Phi_{p,q*}(v'_i(t), 0, 1, 0) |_{t=0}&\in& d(exp_p)_{v_{pq}}(dv_i'/dt (0)) \subset T_qM \\
\tfrac{d}{dt}\Phi_{p,q*}(v_{pq}, t, 1, 0) |_{t=0}&=& v_{pq} \in T_pM \\
\tfrac{d}{dt}\Phi_{p,q*}(v_{pq}, 0, 1+t, 0) |_{t=0}&=& \gamma_{pq}'(1) \in T_qM \\
\tfrac{d}{dt}\Phi_{p,q*}(v_{pq}, 0, 1, w_i(t)) |_{t=0} &\in& v_{pq}^\perp \subset T_pM 
\end{eqnarray}
by our choice of $k,\alpha$ at (\ref{NonDegnerateFoliation}).
Thus by the inverse function theorem, possibly choosing even tighter a collection of 
$\epsilon, \bar{\tau}, \bar{\eta}_i$, we have a diffeomorphism. 

Note further that  pulling back $dV_{g_0}\times dV_{g_0}$ to $\mathcal{N}(p,q)$ we know that 
\begin{align}
    \Phi_{(p,q)}^*(dV_{g_0} \times dV_{g_0})<< dV_{\mathcal{N}_{p,q}} 
\end{align}
where $<<$ denotes absolute continuity of measures. 

Combining this with (\ref{almost-there}) we see that for any $(p,q)\in M\times M$ such that $q$ is not a cut point of $p$
there exists an open set $\mathcal{U}(p,q)$ containing $(p,q)$ such that
\begin{align} 
   \int_{\mathcal{U}(p,q)} |d_j(p',q') - d_0(p',q')| \, dV_{g_0} \times dV_{g_0} \rightarrow 0. \label{IntegralToZero}
\end{align}

Thus we have an open cover of $M\times M \setminus \mathcal{S}$ where
\begin{align}
    \mathcal{S}=  \left\{ \{p\}\times C(p):p\in M\right\},
\end{align}
where $C(p)$ is the cut locus of $p$. Since $\mathcal{S} \subset M\times M$ is a measure zero set we are allowed to proceed with our argument on $(M\times M) \setminus \mathcal{S}$. 
Define the continuous map 
\begin{align}
\Psi: (M\times M) \setminus \mathcal{S} \rightarrow \R, \quad \Psi(p,q) = d_0(q,C(p)),
\end{align}
so that we can define the compact sets  
\begin{align}
K_i=\left\{(p,q) \in (M\times M) \setminus \mathcal{S}: d(q,C(p))\ge \frac{1}{1+i}\right\} \subset (M\times M) \setminus \mathcal{S}
\end{align}
 so that $K_i \subset K_{i+1}$ and 
 \be
 \bigcup_{i=1}^{\infty} K_i = (M\times M) \setminus \mathcal{S}.
 \ee 

 Since
  \begin{align}
 \{\mathcal{U}_{(p,q)}: (p,q) \in (M\times M) \setminus \mathcal{S}\}, 
\end{align}
is an open over of $(M\times M) \setminus \mathcal{S}$, it is also an open cover of $K_i$ for each $i \in \N$. So we can choose a finite subcover $\{U_1,...,U_{I_1}\}$ of $K_1$, and extend to a finite subcover of $K_2$, $\{U_1,...,U_{I_1},...U_{I_2}\}$, and continue in this way to construct a countable collection of elements of $\mathcal{U}$, $\{U_i\}_{i \in \N}$, so that
\begin{align}
(M\times M)\setminus \mathcal{S} \subset \bigcup_{i\in \N} U_i.
\end{align}
By  \eqref{IntegralToZero} we have for all $i\in \N$:
\begin{align}
   \int_{U_i} |d_j(p',q') - d_0(p',q')|  \,dV_{g_0} \times dV_{g_0} \rightarrow 0. \label{IntegralToZero2}
\end{align}

By \eqref{IntegralToZero2} we can choose a subsequence to find pointwise a.e. convergence of $d_j \rightarrow d_0$ on $U_1$. Then we can choose a further subsequence to find pointwise a.e. convergence of $d_j \rightarrow d_0$ on $U_2$ and continue to build a nested sequence of subsequences to obtain pointwise a.e. convergence on each $U_i$. By extracting a diagonal subsequence we then obtain
\begin{align}
    d_{j}(p',q') \paeto d_0(p',q') 
\end{align}
for almost every $(p',q') \in  M\times M$ with respect to $dV_{g_0}\times dV_{g_0}$.

\end{proof}

\subsection{Proof of our Main Theorems }
\label{subsec:MainProof}

We start by proving Theorem \ref{GenGHandFlatConvMetric}.

\begin{proof}
First we note that we can assume $g_j \ge g_0$ by rescaling all the metrics by $(1-1/j)$
and we still have the same hypotheses with the volumes still converging to the volume of $g_0$. In the case where we assume $L^p$ convergence we can use Lemma \ref{LpToLq}, Lemma \ref{LpToLpNorm}, and Lemma \ref{ConfVolBound} to obtain volume convergence.
Applying  Theorem \ref{PointwiseAEConvergenceDistances} we see that, 
after possibly passing to a subsequence, 
$d_j$ converge pointwise to $d_0$ for almost every $(p,q) \in M \times M$.  

By the bi-Lipschitz bounds on $g_j$,
we can apply Theorem \ref{CompactnessThmFlatandGH} to extract a 
further subsequence of $M_j$ such that
\be
(M, d_j) \GHto (M, d_\infty) \textrm{ and } (M, d_j) \Fto (M, d_\infty)
\ee
where $M_\infty=(M, d_\infty)$ where $d_\infty$ is the uniform limit of $d_j$. 
Since we already know $d_j$ converge pointwise to $d_0$ almost everywhere
and both are continuous, we see that $d_\infty=d_0$.  

So for every subsequence there is a further subsequence which converges 
\be
(M, g_j) \GHto (M, g_0) \textrm{ and } (M, g_j) \Fto (M, g_0).
\ee
Thus no subsequence can converge anywhere else or fail to converge, which
implies that the original sequence converges.

Applying the volume convergence hypothesis again combined with
Theorem~\ref{VF-to-mGH-thm} or Theorem~\ref{Global-to-Local}
we have
$(M_j,g_j) \VFto  (M, g_0)$ and $(M_j,g_j) \mGHto  (M, g_0)$.
\end{proof}

We finish by noting that the proof of Theorem \ref{ConfGHandFlatConv} follows from Theorem \ref{GenGHandFlatConvMetric}.

  \bibliographystyle{alpha}
 \bibliography{allen}

\end{document}